\documentclass[a4paper,draft]{article}

\usepackage[dvips]{color} 
\usepackage{lscape}
\usepackage{amsmath,amssymb,amsthm}
\usepackage{amscd}
\usepackage{enumerate}
\usepackage{latexsym}
\usepackage{mathrsfs}
\usepackage[dvips]{graphicx}
\usepackage{chngcntr}
\usepackage{latexsym}
\usepackage[all]{xy}

\newcommand{\Z}{\mathbb{Z}}

\theoremstyle{definition}
\newtheorem{Def}{Definition}[section]
\newtheorem{Rem}[Def]{Remark}

\theoremstyle{theorem}
\newtheorem{Th}[Def]{Theorem}
\newtheorem{Prop}[Def]{Proposition}
\newtheorem{Lem}[Def]{Lemma}

\counterwithin*{Def}{section}

\begin{document}
\title{
The GIT moduli of semistable pairs consisting of a cubic curve and a line on ${\mathbb P}^{2}$
}
\author{Masamichi Kuroda}
\date{}

\maketitle
\begin{abstract}
We discuss the GIT moduli of semistable pairs consisting of a cubic curve and a line on the projective plane. 
We study in some detail this moduli and compare it with another moduli 
suggested by Alexeev. 
It is the moduli of pairs (with no specified semi-abelian action) consisting of a cubic curve with at worst nodal singularities and a line which does not pass through singular points of the cubic curve. 
Meanwhile, we make a comparison between Nakamura's compactification of the moduli of level three elliptic curves and these two moduli spaces. 
\end{abstract}

\section{Introduction}
\footnote[0]{{\hspace{-0.6cm} 2010 \it Mathematics Subject Classification:} 14H10, 14K10. \\
{\it key words and phrases:} Moduli, Stability, Cubic curves}
Let ${\mathbb P}^2$ be the projective plane over an algebraically closed field $k$ of  characteristic not equal to 2 and 3. 
Let $M$ be the set of pairs consisting of a cubic curve and a line on ${\mathbb P}^2$. 
The action of ${\rm PGL} (3)$ on ${\mathbb P}^2$ induces an action of 
${\rm PGL} (3)$ on $M$. 
Let $M^{ss}$ be the set of semistable points of $M$ under this action. 
Then there exists a good categorical quotient of $M^{ss}$ by ${\rm PGL} (3)$, 
which we denote by $\overline{P}_{1,3}$. 
On the other hand, there exists another complete moduli $BP_{1,3}$ suggested by 
\cite{Alexeev02}. 
It is the moduli of pairs $(C, L)$ (with no specified semi-abelian action) such that $C$ is a reduced plane cubic curve with at worst nodal singularities and $L$ is a line which does not pass through singularities of $C$. 
We will construct this moduli by using the 
theory of \cite{KeelMori97}. 
There is also the moduli $SQ_{1,3}$ ($\simeq {\mathbb P}^1$) of Hesse cubic curves defined in \cite{Nakamura99}, which is well-known classically as the modular curve $X (3)$ of level three. 

The purpose of this paper is to study in some detail
the GIT moduli $\overline{P}_{1,3}$ 
following the method of \cite{MFK94}. 
By using the numerical criterion due to Hilbert and Mumford, we can classify unstable, semistable and stable pairs completely 
(see Proposition \ref{list of unstable pairs} and Table \ref{sspair}). 
Moreover by constructing suitable semistable limits, we give nontrivial identifications of semistable pairs in $\overline{P}_{1,3}$ (see Proposition \ref{nontrivial identifications}). 
It is clear that $\overline{P}_{1,3}$ and $BP_{1,3}$ have a common open subset $U_{1,3}$ consisting of pairs $(C, L)$ such that $C$ is a smooth cubic curve and $L$ is a line 
intersecting transversally. 
This enables us to compare 
$\overline{P}_{1,3}$ with $BP_{1,3}$ as follows (see Theorem \ref{main1}): \\
~\\
{\bf Theorem.} {\it 
There exists a birational map 
$f : \overline{P}_{1,3} \rightarrow BP_{1,3}$ such that 
\begin{itemize}
\item the base locus of $f$ is isomorphic to ${\mathbb P}^1$, which is the set of all semistable pairs $(C,L)$ 
consisting of a cuspidal curve $C$ and a line $L$ intersecting transversally at smooth points of $C$, and 
\item the base locus of the birational inverse $f^{-1}$ of $f$ is isomorphic to ${\mathbb P}^1$,  which is the set of all pairs $(C,L)$ 
consisting of a smooth cubic or irreducible nodal cubic curve $C$ and a triple tangent $L$ at a smooth point of $C$. 
\end{itemize}
}

Since $SQ_{1,3}$ is the moduli of Hesse cubics, 
we can define rational maps forgetting the level structure: 
\begin{align*}
\varphi :  SQ_{1,3} \times {\mathbb P}^2 \longrightarrow BP_{1, 3} , \ \ 
\psi : SQ_{1,3} \times {\mathbb P}^2 \longrightarrow \overline{P}_{1, 3} , 
\end{align*}
where ${\mathbb P}^2$ means the space of lines on ${\mathbb P}^2$. 
They are same branched coverings of degree $216$ on the common open subset $U_{1,3}$. 
This enables us to compare $\overline{P}_{1,3}$ and $BP_{1,3}$ with 
$SQ_{1,3} \times {\mathbb P}^2$ as we see Proposition \ref{main2}. 

This paper is organized as follows. 
In Section~2, 
we study the GIT moduli $\overline{P}_{1, 3}$. 
Especially we classify unstable, semistable and stable pairs, respectively, and we give nontrivial identifications of semistable pairs in $\overline{P}_{1, 3}$. 
In Section~3, 
we define the moduli $BP_{1,3}$ by using the Keel-Mori theorem. 
In Section~4, 
we construct a birational map $f : \overline{P}_{1,3} \to BP_{1,3}$ and discuss its properties. 
In Section~5, 
we make a comparison 
$\overline{P}_{1,3}$ and $BP_{1,3}$ with 
$SQ_{1,3} \times {\mathbb P}^2$. 

\section{The GIT moduli $\overline{P}_{1, 3}$} \label{GIT moduli}

We first give the definition of the moduli $\overline{P}_{1,3}$. 
We use mainly definitions and properties in \cite{MFK94}, \cite{Newstead78} and \cite{Dolgachev03}. 
Let $k$ be an algebraically closed field of characteristic not equal to 2 and 3. 
Let $V$ be the dual space of the space of homogeneous polynomials of degree one on the projective plane ${\mathbb P}^{2} = {\rm Proj} (k [x_0, x_1, x_2])$, that is, $V^\vee = k [x_0, x_1 , x_2]_1$. 
Each $F \in (S^3 V)^{\vee} = S^3 V^\vee = k [x_0, x_1, x_2]_3$ (resp. $V^\vee$) corresponds to the cubic curve (resp. line) $V(F)$, where $V(F)$ is the zero set of $F$ in ${\mathbb P}^2$. 
Then we can regard ${\mathbb P} (S^3 V) = {\rm Proj} ({\rm Sym} (S^3 V))$ (resp. ${\mathbb P} (V) = {\rm Proj} ({\rm Sym} (V))$) as the space of cubic curves (resp. lines). 
Hence the set of pairs consisting of a cubic curve and a line on ${\mathbb P}^2$ 
is isomorphic to ${\mathbb P} (S^3 V) \times {\mathbb P} (V)$. 
${\rm SL} (3)$ acts on $S^3 V^\vee$ (resp. $V^\vee$) by 
$g \cdot F (x) = F (g^{-1} \cdot x)$ for all 
$F \in S^3 V^\vee$ (resp. $V^\vee$). 
We have the natural morphism $q : {\rm SL} (3) \rightarrow {\rm PGL}(3) $ which is surjective with a finite kernel. 
For any $g \in {\rm PGL}(3)$, let $\tilde{g} \in {\rm SL} (3)$ be a matrix such that $q (\tilde{g}) = g$, 
and we put 
$g \cdot V(F) := V(\tilde{g} \cdot F)$ for all $F \in S^3 V^\vee$ (resp. $V^\vee$). 
This definition is independent of the choice of $\tilde{g}$. 
Thus we get an action of ${\rm PGL} (3)$ on ${\mathbb P} (S^3 V) \times {\mathbb P} (V)$: 
\begin{eqnarray*} 
{\rm PGL} (3) \times \left( {\mathbb P} (S^3 V) \times {\mathbb P} (V) \right) &\longrightarrow  & {\mathbb P} (S^3 V) \times {\mathbb P} (V) , \\ 
(g, \left( V(F), V(S) \right)) &\longmapsto  & \left( g \cdot V(F), g \cdot V(S) \right) .
\end{eqnarray*} 
Let $\left( {\mathbb P}(S^3 V) \times {\mathbb P}(V) \right)^{ss}$ be the set of 
semistable pairs of ${\mathbb P}(S^3 V) \times {\mathbb P}(V)$ with respect to the action of ${\rm PGL} (3)$. By Theorem 1.10 in \cite{MFK94}, 
there exists a good categorical quotient 
\begin{align*}
\phi : \left( {\mathbb P}(S^3 V) \times {\mathbb P}(V) \right)^{ss} \longrightarrow \overline{P}_{1, 3} := \left( {\mathbb P}(S^3 V) \times {\mathbb P}(V) \right)^{ss} /\!/ {\rm PGL} (3) .
\end{align*}

Next to classify unstable, semistable and stable pairs of ${\mathbb P} (S^3 V) \times {\mathbb P} (V)$ under the action of ${\rm PGL} (3)$, we give a quick review on the numerical criterion due to Hilbert and Mumford. 

\begin{Def}
An {\em one parameter subgroup} of ${\rm PGL} (3)$ is a nontrivial homomorphism of algebraic groups $\lambda : {\mathbb G}_m \rightarrow {\rm PGL} (3)$, and it is {\em normalized} if 
\begin{align*} 
\lambda(t)={\rm Diag} (t^{r_0},t^{r_1},t^{r_2})
\end{align*}
for some $r_i \in \Z$ with $r_0 \geq r_1 \geq r_2$ and $ r_0 + r_1 + r_2 = 0$. 
In what follows, in this paper, 
we denote by $1$-PS the one parameter subgroup of ${\rm PGL} (3)$. 
\end{Def}
To analyze stability we use normalized $1$-PS's. 
Let 
\begin{equation} \label{FS}
F (x) = \sum_{0 \leq i, j, i+j \leq 3} a_{i j} x_0^{3-i-j} x_1^i x_2^j, 
 \ \ S (x) =  b_0 x_0+b_1 x_1+b_2 x_2 . 
\end{equation}
Then the image of $z = (V(F), V(S)) \in {\mathbb P} (S^3 V) \times {\mathbb P} (V)$ under the Segre embedding 
${\mathbb P} (S^3 V) \times {\mathbb P} (V) \hookrightarrow {\mathbb P} (S^3 V \otimes V)$ is $V(H)$, where 
$$
H(x, y) := F(x) S(y) = \sum a_{i j} b_k x_0^{3 - i - j} x_1 ^i x_2^j y_k .
$$
The action of ${\rm PGL} (3)$ on ${\mathbb P} (S^3 V) \times {\mathbb P} (V)$ is extended to an action on ${\mathbb P} (S^3 V \otimes V)$. 
For any normalized $1$-PS $\lambda $, we have 
$$
\lambda (t) \cdot V(H) = V(H (\lambda (t)^{-1} \cdot x, \lambda (t)^{-1} \cdot y)) = V \left( \sum a_{i j} b_k t^{-R_{ijk}}x_0^{3 - i - j} x_1 ^i x_2^j y_k \right)
$$
where $R_{i j k } = (3-i-j)r_0+ir_1+jr_2+r_k$. 
Then we define 
$$
\mu(z,\lambda) := \max\{ R_{i j k} \mid a_{ij}b_k\neq 0\}. 
$$
By Theorem 2.1 in \cite{MFK94}, we have the following criterion: 

\begin{Th} \label{criterion}
For any $z \in {\mathbb P} (S^3 V) \times {\mathbb P} (V)$, 
$z$ is semistable (resp. stable) if and only if $\mu (g \cdot z, \lambda ) \geq 0$ (resp. $> 0$) for any normalized $1$-PS $\lambda $ and any $g \in {\rm PGL} (3)$. 
Note that $z$ is called unstable if $z$ is not semistable. 
\end{Th}

To calculate $\mu (z, \lambda )$ we study relations between $30$ integer numbers 
$R_{ijk}$. 
By simple calculations, we obtain following Lemma: 
\begin{Lem} \label{Rijk}
For any $r_i \in {\mathbb Z}$ with $r_0 \geq r_1 \geq r_2$ and $ r_0 + r_1 + r_2 = 0$, 
we have that 
\begin{equation*} 
\begin{array}{c c c c c c c c c} 
R_{000} & & & & & & & & \\
\rotatebox{-90}{$\geq$ } & & & & & & & & \\
R_{001}=R_{100} & \! \geq \! & R_{002}=R_{010} & & & & & & \\
\rotatebox{-90}{$\geq$ } & &\rotatebox{-90}{$\geq$ } & & & & & & \\
R_{101}=R_{200} & \! \geq \! & R_{102}=R_{110} & \! \geq \! & R_{012}=R_{020} & & & & \\
 & & =R_{011} & & & & & & \\
\rotatebox{-90}{$\geq$ } & &\rotatebox{-90}{$\geq$ } & &\rotatebox{-90}{$\geq$ } & & & & \\
R_{201}=R_{300} & \! \geq \! & R_{202}=R_{111} & \! \geq \! & R_{112}=R_{021} & \! \geq \! & R_{022}=R_{030} & &\\
& & =R_{210} & & =R_{120} & & & &  \\
\rotatebox{-90}{$\geq$ } & &\rotatebox{-90}{$\geq$ } & &\rotatebox{-90}{$\geq$ } & &\rotatebox{-90}{$\geq$ } & & \\
R_{301} & \! \geq \! & R_{302}=R_{211} & \! \geq \! & R_{212}=R_{121} & \! \geq \! & R_{122}=R_{031} & \! \geq \! & R_{032}. \\
\end{array} 
\end{equation*}
\end{Lem}

\subsection{Unstable pairs}

We classify all pairs $z = (C, L)$ with $\mu(z,\lambda)<0$ 
for some normalized $1$-PS $\lambda$, 
where $C = V(F)$, $L = V(S)$, and $F$ and $S$ are given by (\ref{FS}). 
By Theorem \ref{criterion}, we know that any unstable pair is equivalent to one of such pairs under the action of ${\rm PGL} (3)$. 
Let $\displaystyle F_\ell (x_1, x_2) = \sum_{i+j=\ell} a_{ij} x_1^i x_2^j$. Then 
$F=F_3+x_0F_2+x_0^2F_1+x_0^3F_0$. 
\begin{Lem} \label{lemma1}
Assume $\mu(z,\lambda) < 0$ for some normalized $1$-PS $\lambda$. Then 
\begin{enumerate} [\rm (1)]
\item $a_{00} = a_{10} = b_0 = b_1 = 0$, 
\item $a_{00} = a_{10} = a_{01} = b_0 = 0$, or 
\item $a_{00} = a_{10} = a_{01} = a_{20} = a_{11} = 0$. 
\end{enumerate} 
\end{Lem}

\begin{proof}
Since $\mu(z,\lambda) < 0$
for some normalized $1$-PS $\lambda$ and we have $R_{102} = r_0 \geq 0$, we obtain 
$a_{00} b_0 = a_{00} b_1 = a_{10} b_0 = a_{00} b_2 = a_{01} b_0 = a_{10} b_1 = a_{20} b_0 = a_{10} b_2 = a_{11} b_0 = a_{01} b_1 = 0$ by Lemma \ref{Rijk}. 
Then if $b_0 \not = 0$ (resp. $b_0 = 0$ and $b_1 \not = 0$, or $b_0 = b_1 = 0$ and $b_2 \not = 0$), then we obtain (3) (resp. (2), or (1)). 
\end{proof}

\begin{Prop}\label{list of unstable pairs}
The pair $(C, L)$ is unstable 
if and only if 
one of the following is true: 
\begin{enumerate} [\rm (i)]
\item $L$ is a triple tangent to $C$,
\item $L$ is contained in $C$,
\item $L$ passes through a double point of $C$, 
\item $C$ has a triple point, 
\item $C$ is nonreduced.
\end{enumerate} 
\end{Prop}

\begin{proof}
We first prove the only if part. 
Let $z = (C, L)$ be unstable. Then we have $\mu(z,\lambda) < 0$ for some normalized $1$-PS $\lambda$. 
Hence (1), (2) or (3) in Lemma \ref{lemma1} is true. 

When (1) is true, we may assume that $a_{01} \not = 0$, since if $a_{01} = 0$ then we obtain the case (2). 
If $a_{20} \not = 0$, then by Lemma \ref{Rijk}, 
\begin{align*}
0 > \mu(z,\lambda) &= {\rm max} \{ R_{012}, R_{202} \} = {\rm max} \{ -2 r_1,  r_1 \} ,
\end{align*}
which is absurd. This shows $a_{20}=0$. 
Hence we obtain 
$
C : x_2 (a_{21} x_1^2 + a_{12} x_1 x_2 + a_{03} x_3^2 + a_{11} x_0 x_1 + a_{02} x_0 x_2 + a_{01} x_0^2) + a_{30} x_1^3 = 0
$
and $L : x_2 = 0$. 
If $a_{30}\neq 0$, then $L$ is a triple tangent to $C$ at $(1 : 0 : 0)$, 
and if $a_{30} = 0$, then $L$ is contained in $C$. 

When (2) is true, 
$L : b_1 x_1 + b_2 x_2 = 0$ passes through a double point $(1 : 0 : 0)$ of $C : F_3 + x_0 F_2 = 0$. 

When (3) is true, we may assume that $b_0 \not = 0$, since if $b_0 = 0$ then we obtain the case (2). 
If $a_{02} = 0$, then $C : F_3 = 0$. Hence $C$ has a triple point $(1 : 0 : 0)$. 
Let $a_{02} \not = 0$. If $a_{30}$ or $a_{21} \not = 0$, then by Lemma \ref{Rijk}, 
\begin{align*}
0 > \mu(z,\lambda) \geq  {\rm max} \{ R_{020}, R_{210} \} = {\rm max} \{ -2 r_1, r_1 \} ,
\end{align*}
which is absurd. Hence $a_{30} = a_{21} = 0$. Thus 
$C : x_2^2 (a_{02} x_0 + a_{12} x_1 + a_{03} x_2 ) = 0$ with $a_{02} \not = 0$, and hence $C$ is nonreduced. 

Next we prove the if part. 
When (i) or (ii) is true, 
there exists some $g \in {\rm PGL} (3)$ such that 
$g \cdot L : x_2 = 0$ and 
$g \cdot C : x_2 A + a_{30} x_1^3 = 0$ for some quadratic $A$.  
In particular, 
$b_0 = b_1 =0,b_2\neq 0$ and $a_{00} = a_{10} = a_{20}=0$. 
Then by Lemma \ref{Rijk}, we obtain 
\begin{align*}
\mu(g \cdot z ,\lambda) &\leq  \max \{ R_{012}, R_{302} \} = \max \{ - 2 r_1, 2 r_1 - r_0 \} = -1
\end{align*}
for $r = (3, 1, -4)$. Hence $z = (C, L)$ is unstable by Theorem \ref{criterion}. 

When (iii) is true, 
there exists some $g \in {\rm PGL} (3)$ such that 
$g \cdot L : x_2 = 0$ 
and $g \cdot C : x_0F_2+F_3 = 0$. 
In particular, $a_{00} = a_{10} = a_{01} = 0$ and $b_0 = b_1 = 0$. 
Then by Lemma \ref{Rijk}, we obtain 
\begin{align*}
\mu (g \cdot z, \lambda ) \leq R_{202} = r_0 + 2r_1 + r_2 = r_1 = -1
\end{align*}
for $r = (2, -1, -1)$, and hence $z = (C,L)$ is unstable by Theorem \ref{criterion}. 

When (iv) is true, there exists some $g \in {\rm PGL} (3)$ such that 
$g\cdot C : F_3 = 0$. In particular, $F_0 = F_1 = F_2 = 0$. 
Then by Lemma \ref{Rijk}, we obtain 
$$\mu (g \cdot z, \lambda ) \leq R_{3 0 0} = 3 r_1 + r_0 = -1$$
for $r = (2, -1, -1)$, and hence $z = (C, L)$ is unstable by Theorem \ref{criterion}. 

When (v) is true, there exists some $g \in {\rm PGL} (3)$ such that 
$g\cdot C : x_2^2 (a_{02} x_0 + a_{12} x_1 + a_{03} x_2 ) = 0$. In particular, 
$a_{00} = a_{10} = a_{01} = a_{20} = a_{11} = a_{30} = a_{21} = 0$. 
Then by Lemma \ref{Rijk}, we obtain 
$\mu (g \cdot z, \lambda ) \leq R_{0 1 2} = -2 r_1 = - 2$ 
for $r = (1, 1, -2)$, and hence $z = (C, L)$ is unstable by Theorem \ref{criterion}. 
\end{proof}

\subsection{Semistable and stable pairs}

From Proposition \ref{list of unstable pairs}, 
we obtain 
\begin{Prop} \label{all semistable pairs}
The pair $(C, L)$ is semistable 
if and only if 
any of the following is true: 
\begin{enumerate} [{\rm (i)}]
\item $C$ is reduced and 
does not have a triple point, 
\item $L$ is not contained in $C$,
\item $L$ does not pass through any double point of $C$, 
\item $L$ is not a triple tangent to $C$. 
\end{enumerate}
\end{Prop}

Next we classify stable pairs. For it, we classify 
semistable pairs $z = (C, L)$ with $\mu(z,\lambda)=0$ 
for some nontrivial normalized $1$-PS $\lambda$. 
In this case, $z$ is semistable but not stable by Theorem \ref{criterion}. 

\begin{Def}
We say that   
{\em $L$ is $2$-tangent to $C$} if 
$L$ is tangent to an irreducible cubic $C$, 
but not triply tangent, at a smooth point of $C$, and 
{\em $L$ is $3$-tangent to $C$} if $L$ is triply tangent to $C$.
\end{Def}

\begin{Prop} \label{ssnots}
Let $(C, L)$ be semistable. Then $(C, L)$ is not stable 
if and only if 
one of the following is true: 
\begin{enumerate} [\rm (i)]
\item $L$ is $2$-tangent to $C$, 
\item $C$ is an irreducible conic $Q$ plus a line $L'$, and $L'$ is tangent to $Q$. 
\end{enumerate} 
\end{Prop}

\begin{proof}
We first prove the only if part. 
Let $z = (C, L)$ be semistable but not stable. Then we have $\mu(z,\lambda) = 0$ for some nontrivial normalized $1$-PS $\lambda$. 
Since $R_{102} = r_0 > 0$, by same arguments as the proof 
of Lemma \ref{lemma1}, we obtain (1), (2) or (3) in Lemma \ref{lemma1}. 
However, in the case (2), $L$ passes through a double point of $C$. Hence $(C, L)$ is unstable, which is absurd. 
Thus (1) or (3) is true. 

When (1) is true, if $a_{01}$ or $a_{20} = 0$, then $(C, L)$ is unstable by the proof of Proposition \ref{list of unstable pairs}. 
Hence $a_{01} a_{20} \not = 0$. 
Then $C \cap L = (x_2 = 0, x_1^2 (a_{20} x_0 + a_{30} x_1) = 0)$, and hence 
$L$ is $2$-tangent to $C$ at $(1 : 0 : 0)$. 

When (3) is true, if $b_0$ or $a_{02} = 0$, then $(C, L)$ is unstable by the proof of Proposition \ref{list of unstable pairs}. 
Thus $a_{02} b_0 \not = 0$. If $a_{30} \not = 0$, then by Lemma \ref{Rijk}, 
$$
0 = \mu (z , \lambda) = {\rm max} \{ R_{020}, \ R_{300} \}={\rm max} \{ -2r_1, r_0 + 3 r_1 \} > 0, 
$$
which is absurd. Hence $a_{30} = 0$. If $a_{21} = 0$, then $C$ is nonreduced. Hence $(C, L)$ is unstable, which is absurd. 
Thus $a_{21} \not = 0$. Then $C : a_{21} x_1^2 x_2 + a_{12} x_1 x2^2 + a_{03} x_2^3 + a_{02} x_0 x_2^2 = 0$ with $a_{21} a_{02} \not = 0$, 
and hence $C$ is an irreducible conic $Q : a_{21} x_1^2 + x_2 (a_{02} x_0 + a_{12} x_1 + a_{03} x_2 )=0$ plus a line $L' : x_2 = 0$, and $L'$ is tangent to $Q$ at $(1 : 0 : 0)$. 

Next we prove the if part. 
In the case (i), 
there exists some $g \in {\rm PGL} (3)$ such that $g\cdot L : x_2 = 0$ and 
$g \cdot C : a_{01}x_0^2x_2+A(x)x_2+ x_1^2 (a_{20}x_0+a_{30}x_1) = 0$ with some 
quadratic $A(x)$ of degree at most one in $x_0$. 
Since $g \cdot L$ does not pass 
through any singular points of $g \cdot C$, 
$(1 : 0 : 0)$ is not a singular point of $g \cdot C$ and hence $a_{01}\neq 0$.
Since $g \cdot L$ is not $3$-tangent to $g \cdot C$, we have $a_{20}\neq 0$. 
Therefore $a_{00} = a_{10} = b_0 = b_1 = 0$ and $a_{01} a_{20} \not = 0$. 
By Lemma \ref{Rijk}, we obtain 
\begin{align*}
\mu( g \cdot z,\lambda) = \max\{R_{012}, R_{202} \} = \max \{ - 2 r_1, r_1 \} = 0
\end{align*}
for $r=(1, 0, -1)$. 
Hence $z$ is not stable by Theorem \ref{criterion}. 

In the case (ii), there exists some $g \in {\rm PGL} (3)$ such that 
$g \cdot L:x_0=0$ and $g \cdot C = Q + L'$, where $Q:x_2(a_{02}x_0+a_{12}x_1+a_{03}x_2)+a_{21} x_1^2=0$ and $L' : x_2=0$. 
Since $Q$ is irreducible, we have $a_{02} a_{21} \neq 0$. 
In particular, $a_{00} = a_{10} = a_{01} = a_{20} = a_{11} = a_{30} = 0$ and $a_{21} a_{02} b_0 \neq 0$. 
Then by Lemma \ref{Rijk}, we obtain 
$\mu(g \cdot z, \lambda) = \max \{ R_{020}, R_{210} \}
=\max \{-2 r_1, r_1\} = 0$ for $r=(1,0,-1)$, and hence 
$z$ is not stable by Theorem \ref{criterion}. 
\end{proof}

By Proposition \ref{ssnots}, 
we obtain the complete classification of semistable and stable pairs, which is given in Table \ref{sspair} below. 
We denote by $S_k \subset {\mathbb P}(S^3 V) \times {\mathbb P}(V)$ 
the locus of semistable pairs in the column $k$ in Table \ref{sspair}. 
\begin{table}[htb]
\begin{center}
\caption{The semistable and stable pairs $(C, L)$} \label{sspair}
\begin{tabular}{|c|c|c|c|}
\hline
\multicolumn{1}{|c|}{{$C$}} &\multicolumn{1}{|c|}{ $C$ and $L$} &\multicolumn{1}{|c|}{{\rm stability}} & \multicolumn{1}{|c|}{{$k$}} \\ \hline \hline
smooth & $L$ : transv. to $C$ & stable & $1$ \\ \cline{2-4}
 & $L$ : $2$-tangent to $C$ & semistable & $2$\\ \hline
$3$-gon & $L$ : transv. to $C$ & stable & $3$\\ \hline
line $L'$ $+$ conic $Q$ & $L, L'$ : transv. to $Q$ & stable & $4$\\ \cline{2-4}
 & $L$ : tangent to $Q$ & semistable & $5$\\ \cline{2-4}
 & $L'$ : tangent to $Q$ & semistable & $6$\\ \cline{2-4}
 & $L, L'$ : tangent to $Q$ & semistable & $7$\\ \hline
 irred. a node & $L$ : transv. to $C$ & stable & $8$\\ \cline{2-4}
 & $L$ : $2$-tangent to $C$ & semistable & $9$\\ \hline
 irred. a cusp & $L$ : transv. to $C$ & stable & $10$\\ \cline{2-4}
 & $L$ : $2$-tangent to $C$ & semistable & $11$\\ \hline
\end{tabular}
\end{center}
\end{table}

\subsection{Nontrivial identifications of semistable pairs in $\overline{P}_{1, 3}$}

We give nontrivial identifications of semistable pairs in the GIT moduli
$$
\phi : \left( {\mathbb P}(S^3 V) \times {\mathbb P}(V) \right)^{ss} \longrightarrow \overline{P}_{1, 3} = \left( {\mathbb P}(S^3 V) \times {\mathbb P}(V) \right)^{ss} /\!/ {\rm PGL} (3) .
$$

\begin{Lem} \label{loci}
We define semistable pairs $z_i \in S_i$ ($i = 3, 5, 6, 7, 11$) as follows: 
\begin{gather*}
z_3 = \left( V( x_0 x_1 x_2), V (x_0+x_1+x_2) \right), \ 
z_5 = \left( V( x_0 (x_0 x_2 + x_1 (x_1 + x_2 ) )) , V( x_2) \right), \\
z_6 = \left( V( x_0 (x_0 x_2 + x_1 (x_0 + x_1 ) ) ), V( x_2) \right), \ 
z_7 = \left( V( x_0 (x_0 x_2 + x_1^2 ) ), V ( x_2 ) \right), \\
z_{11} = \left( V (x_0^2 x_2 + x_1^2 (x_0 + x_1) ), V( x_2 ) \right) .
\end{gather*}
Then we have 
$\displaystyle 
S_i = \left\{ 
\begin{array}{ll}
{\rm PGL} (3) \cdot z_i & (i = 3, 7, 11), \\
{\rm PGL} (3) \cdot z_i \cup {\rm PGL} (3) \cdot z_7 & (i = 5, 6). \\
\end{array}
\right.$
\end{Lem}

\begin{proof}
Let $(C, L)$ be any semistable pair in $S_i$ ($i \in \{ 3, 5, 6, 7, 11\}$). 

When $i=3$, there exists some $g_1 \in {\rm PGL} (3)$ such that $g_1\cdot C = V(x_0 x_1 x_2)$, since $C$ is a $3$-gon. 
We write $g_1 \cdot L = V(b_0 x_0 + b_1 x_1 + b_2 x_2)$. Since $g_1 \cdot L$ is transversal to $g_1\cdot C$, we have $b_0 b_1 b_2 \not = 0$. 
Let $g_2 := {\rm Diag} (b_0, b_1, b_2) \in {\rm PGL} (3)$. 
Then we obtain $g_2 g_1 \cdot (C, L) = z_3$. 

When $i = 5$, there exists some $g_1 \in {\rm PGL} (3)$ such that 
$g_1 \cdot C = g_1 \cdot Q + g_1 \cdot L'$, 
$g_1\cdot L' = V(x_0)$ and $g_1 \cdot L = V(x_2)$. 
We write 
\begin{align*}
g_1 \cdot Q = V \left (a_{20} x_1^2 + a_{11} x_1 x_2 + a_{02} x_2^2 + a_{10} x_0 x_1 + a_{01} x_0 x_2 + a_{00} x_0^2 \right). 
\end{align*}
Since $g_1\cdot L \cap g_1\cdot Q = V (x_2 , a_{20} x_1^2 + a_{10} x_0 x_1 + a_{00} x_0^2)$ is a double point, we have $a_{10}^2 = 4 a_{20} a_{00}$. 
If $a_{20} = 0$, then $a_{10} = 0$, and hence $g_1 \cdot C$ has a double point $(0 : 1 : 0)$. 
Then $g_1 \cdot L$ passes through this point, which is absurd. 
Thus we get $a_{20} \not = 0$. 
Then there exists some $g_2 \in {\rm PGL} (3)$ such that $g_2 g_1 \cdot L' = V(x_0)$, $g_2 g_1 \cdot L = V(x_2)$ and 
$
g_2 g_1 \cdot Q = V \left (x_1^2 + x_2 ( a'_{01} x_0 + a'_{11} x_1)  \right)
$ for some $a'_{01}$ and $a'_{11} \in k$. 
Since $g_2 g_1 \cdot Q$ is irreducible, $a'_{01} \not = 0$. 
Thus we have 
\begin{align*}
g\cdot C = V(x _0 (x_0 x_2 + x_1 (x_1 + a_{11} x_2))) \ \ {\rm and} \ \ g\cdot L = V(x_2), 
\end{align*} 
where $g = g_3 g_2 g_1$ and $g_3 = {\rm Diag} (a_{01}, 1, 1) \in {\rm PGL} (3)$. 
If $a_{11}=0$, then $g \cdot (C, L) = z_7$, and if $a_{11}\not = 0$, then 
$g_4 g \cdot (C, L) = z_5$, where $g_4 = {\rm Diag} (1/a_{11}, 1, a_{11})$. 

When $i=6$ (resp. $i=7$), by similar arguments as above, we obtain  
$$
g\cdot C = V(x _0 (x_0 x_2 + x_1 (b x_0 + x_1))) \ \ {\rm and} \ \ g \cdot L = V(x_2)
$$ 
(resp. $g \cdot (C, L) = z_7$) for some $g \in {\rm PGL} (3)$. 
If $b=0$, then $g \cdot (C, L) = z_7$, 
and if $b \not = 0$, then 
$g' g \cdot (C, L)=z_6$, where $g' = {\rm Diag} (b, 1, 1/b)$. 

When $i=11$, there exists some $g_1 \in {\rm PGL} (3)$ such that $g_1 \cdot C = V(x_0 ^2 x_2 + x_1^3)$, since $C$ is a cuspidal curve. 
We write $g_1 \cdot L = V(b_0 x_0 + b_1 x_1 + b_2 x_2)$. 
Since $g_1 \cdot L$ does not passes through the cusp point $(0 : 0 : 1)$ of $g_1\cdot C$, we may assume $b_2 = 1$. 
Let $g_2 \in {\rm PGL} (3)$ sending $g_2^{-1} : (x_0, x_1, x_2) \mapsto (x_0, x_1, -b_0 x_0 - b_1 x_1 + x_2)$. Then we have 
\begin{eqnarray*}
g_2 g_1 \cdot C = V( x_0^2 x_2 + x_1^3 - b_1 x_0^2 x_1 - b_0 x_0^3 ) 
\end{eqnarray*}
and $g_2 g_1 \cdot L = V(x_2)$. 
Since $g_2 g_1 \cdot L$ is $2$-tangent to $g_2 g_1 \cdot C$, we have $x_1^3 - b_1 x_0^2 x_1 - b_0 x_0^3 = (x_1 - a x_0)^2 (x_1 - b x_0)$ for some $a \not = b$. 
Hence $b_0 = - 2 a^3$, $b_1 = 3 a^2$ and $b = - 2 a$. In particular, $a \not = 0$. 
Let $g_3 \in {\rm PGL} (3)$ sending $g_3^{-1} : (x_0, x_1, x_2) \mapsto (x_0 / (3 a ), x_0 /3 + x_1, 9 a^2 x_2)$ and put $g = g_3 g_2 g_1 \in {\rm PGL} (3)$. 
Then we have $g \cdot (C, L) = z_{11}$
\end{proof}

\begin{Prop} \label{nontrivial identifications}
$\phi(S_3)$, $\phi(S_5)$, $\phi(S_6)$, $\phi(S_7)$ and $\phi(S_{11})$ are single points in $\overline{P}_{1, 3}$ and 
$\phi(S_5)=\phi(S_6)=\phi(S_7) = \phi(S_{11})$. 
\end{Prop}

\begin{proof}
By Lemma \ref{loci}, $\phi(S_3)$, $\phi(S_7)$ and $\phi(S_{11})$ are single points clearly. 
We prove $\phi(S_i) = \phi(S_7)$ ($i=5$, $6$, $11$). 
For it, by Lemma \ref{loci}, we have to show $\phi(z_i) = \phi(z_7)$ for each $i=5$, $6$, $11$, respectively. 
It is equivalent to 
\begin{eqnarray} \label{i=7}
\overline{{\rm PGL} (3) \cdot z_i} \cap \overline{{\rm PGL} (3) \cdot z_7 } \cap ({\mathbb P} (S^3 V) \times {\mathbb P} (V))^{ss} \not = \emptyset
\end{eqnarray}
by Theorem 3.14, (iii) in \cite{Newstead78}. 
For any $ a \in k ^{\times}$, 
let $g_a = {\rm Diag} (a,1,1/a) \in {\rm PGL} (3)$ and 
$h_a  = {\rm Diag} (1, 1/a, 1/a^2) \in {\rm PGL} (3)$. 
Then we have 
\begin{align*}
g_a \cdot z_5 &= \left( V( x_0 (x_0 x_2 + x_1 (x_1 + a x_2 ) )) , V( x_2) \right) \in ({\mathbb P} (S^3 V) \times {\mathbb P} (V))^{ss}, \\
g_a^{-1} \cdot z_6 &= \left( V( x_0 (x_0 x_2 + x_1 (a x_0 + x_1 ) )) , V( x_2) \right) \in ({\mathbb P} (S^3 V) \times {\mathbb P} (V))^{ss}, \\
h_a \cdot z_{11} &= \left( V( x_0^2 x_2 + x_1^2 (x_0 + a x_1 ) )) , V( x_2) \right) \in ({\mathbb P} (S^3 V) \times {\mathbb P} (V))^{ss}. 
\end{align*}
Hence 
$\underset{a \rightarrow 0}{\lim} \ g_a \cdot z_5 
= \underset{a \rightarrow 0}{\lim} \ g_a^{-1} \cdot z_6
= \underset{a \rightarrow 0}{\lim} \ h_a \cdot z_{11} = z_7$, 
and hence we obtain (\ref{i=7}) for each $i=5$, $6$, $11$, respectively. 
\end{proof}

\begin{Lem} \label{Cusp}
We define $W_{\rm Cusp}:=\phi( S_{10}) \cup \phi(S_{11} ) \subset \overline{P}_{1, 3}$. 
Then $W_{\rm Cusp} \simeq {\mathbb P}^1$ in $\overline{P}_{1,3}$, which is covered with two affine subsets
\begin{align*}
U^{\rm Cusp}_1&=\{ (C,L) \mid C : x_0 x_2^2 = x_1^3 , L : x_0 = b x_1 + x_2 \} \simeq {\rm Spec} \ k[b^3],\\
U^{\rm Cusp}_2&=\{ (C,L) \mid C : x_0 x_2^2 = x_1^3 , L : x_0 = x_1 + c x_2 \} \simeq {\rm Spec} \ k[c^2],
\end{align*}
where $b^3$ is identified with $1/c^2$.
\end{Lem}

\begin{proof}
For any stable pair $(C,L)\in S_{10}$, there exists $g \in {\rm PGL} (3)$ such that 
$$g \cdot C : x_0x_2^2=x_1^3, \ g \cdot L : x_0=b_1x_1+b_2x_2 $$
with $(b_1, b_2) \not = (0, 0)$. 
Then the discriminant $D=4b_1^3-27b_2^2$ is nonzero, since $g \cdot L$ is transversal to $g \cdot C$. 
Since any linear automorphism of $g \cdot C$ keeping the cusp stable is of the form 
$$
(x_0,x_1,x_2)\mapsto (x_0, \lambda^2x_1, \lambda^3x_2), 
$$
$\phi(S_{10})$ is the set 
$\{  (b_1, b_2) \mid D \not = 0 \} / \sim $, where the equivalence relation $(b_1, b_2) \sim (c_1, c_2)$ is defined by $(c_1, c_2) = (\lambda ^2 b_1, \lambda ^3 b_2)$ for some $\lambda \in k^{\times}$. 
Thus $\phi(S_{10})$ is isomorphic to the weighted homogeneous space ${\mathbb P} (2,3)$ minus 
a single point defined by the discriminant $D=4b_1^3-27b_2^2=0$. 
This point corresponds to $\phi(S_{11})$. 
In fact, if $4 b_1^3 = 27 b_2^2$, then $g \cdot L$ is $2$-tangent to $g \cdot C$. 
Thus we obtain $W_{\rm Cusp}=\phi(S_{10}) \cup \phi(S_{11}) \simeq  {\mathbb P} (2, 3) \simeq {\mathbb P}^1$, 
which is covered with $U^{\rm Cusp}_1$ and $U^{\rm Cusp}_2$ defined as above. 
\end{proof}

\section{The moduli space $BP_{1,3}$}

This section is mainly due to contributions by Iku Nakamura. 
We give the definition of the moduli space $BP_{1,3}$. 
Its existence is suggested by \cite{Alexeev02}. 
It is the moduli space of pairs $(C, L)$ (with no specified semi-abelian action) consisting of a reduced cubic curve $C$ with at worst nodal singularities and a line $L$ which does not pass through singularities of $C$. 
We construct this moduli as follows: 
Let $W$ be an open subscheme of ${\mathbb P}(S^3 V) \times {\mathbb P}(V)$ consisting of such pairs. 
Then $W$ is invariant under the action of ${\rm PGL} (3)$ on ${\mathbb P}(S^3 V) \times {\mathbb P}(V)$ defined in Section \ref{GIT moduli}. 
By the Keel-Mori theorem (Corollary 1.2 in \cite{KeelMori97}), 
the quotient $W / {\rm PGL} (3)$ exists as a separated algebraic space over $k$, 
if the action of ${\rm PGL} (3)$ is proper and the stabilizer of any pair in $W$ is finite. 
Thus the following Lemma proves that the quotient $W / {\rm PGL} (3)$ exists as a  separated algebraic space, which we denote by $BP_{1,3}$. 
\begin{Lem} \label{key of BP13}
Let $G = {\rm PGL} (3)$. 
\begin{enumerate}[(1)]
\item The action of $G$ is proper, in other words, the morphism 
$\pi : G \times_k W \to W \times_k W$ sending $(g, x) \mapsto (g \cdot x, x)$ is proper. 
\item For any pair $(C, L) \in W$, the stabilizer ${\rm Stab}_{G} (C, L)$ is finite. 
\end{enumerate}
\end{Lem}

\subsection{Proof of (1) in Lemma \ref{key of BP13}}

We prove (1) in Steps 1-3. 

{\bf Step 1.} 
We prove the following lemma. 

\begin{Lem} \label{cha. of properness}
Let $W$ be an algebraic variety of finite type over $k$, 
$G$ be an algebraic group variety over $k$ acting on $W$, and 
$H$ be a complete variety over $k$ which contains $G$ as a Zariski open dense subset. 
Let $\pi : G \times _k W \to W \times_k W$ be a morphism sending 
$(g, v) \mapsto (g \cdot v, v)$, 
$\Delta$ be the graph of $\pi$, 
and $\overline{\Delta}$ be the closure of $\Delta$ in 
$H \times_k W \times_k W \times _k W$ with reduced structure. 
Then $\pi$ is proper if and only if $\Delta = \overline{\Delta}$. 
\end{Lem}

\begin{proof}
Let 
\begin{gather*}
p_{1, 2} \ ({\rm resp.} \ p_{3, 4}) : H \times_k W \times_k W \times _k W \longrightarrow H \times_k W \ ({\rm resp.} \  W \times_k W )
\end{gather*}
be the projection to the $(1, 2)$ (resp. $(3,4)$)-component. 

First we prove the only if part. 
Assume that $\pi$ is proper and $\Delta \ne \overline{\Delta}$. 
Let $x \in \overline{\Delta} (k) \setminus \Delta (k)$. 
Since $\Delta$ is dense in $\overline{\Delta}$, we can choose a CDVR (complete discrete valuation ring) $R$ and a morphism $\alpha : {\rm Spec} \ R \to \overline{\Delta}$ 
such that 
\begin{align*}
\alpha ({\rm Spec} \ k (0)) = x \ \ \mbox{and} \ \ \alpha ({\rm Spec} \ K) \subset \Delta, 
\end{align*}
where $k(0)$ is the residue field of $R$ and $K$ is the fraction field of $R$. 
Let $f := p_{3,4} \circ \alpha$ and $h := p_{1, 2} \circ \alpha$. 
Since $\alpha ({\rm Spec} \ K) \subset \Delta$, we have 
\begin{align*}
h_K : {\rm Spec} \ K \longrightarrow G \times_k W \ \ 
\mbox{and} \ \ \pi \circ h_K = f_K. 
\end{align*}
Since $\pi$ is proper, by the valuative criterion of properness (\cite{EGA}, II, (7.3.8)), 
there exists a morphism 
$
\phi : {\rm Spec} \ R \rightarrow G \times_k W
$ 
such that $\phi_K = h_K$ and $\pi \circ \phi = f$. 
It follows $h = \phi$, so $h ({\rm Spec} \ R) = \phi ({\rm Spec} \ R) \subset G \times_k W$ 
and $\pi \circ h = f$. 
This shows $\alpha ({\rm Spec} \ R) \subset \Delta$. 
This contradicts $\alpha ({\rm Spec} \ k (0)) = x \not \in \Delta (k)$. 
Hence $\Delta = \overline{\Delta}$. 

Next we prove the if part. 
Let $R$ be any DVR (discrete valuation ring), any morphism 
$f : {\rm Spec} \ R \to W \times_k W$ and any morphism 
$h : {\rm Spec} \ K \to G \times_k W$ such that $\pi \circ h = f_K$. 
Then we can write 
$f = (x, y) \in (W \times_k W )(R)$ and $h = (g, y_K) \in (G \times_k W) (K)$ with 
$g \cdot y_K = x_K$. 
Since $g \in G (K)$ extends to $\tilde{g} \in H (R)$, we have 
$(\tilde{g}, y, x, y) \in (H \times_k W \times_k W \times_k W) (R)$ 
and 
$(\tilde{g}_K, y_K, x_K, y_K) = (g, y_K, x_K, y_K) 
= (g, y_K, g \cdot y_K, y_K) \in \Delta (K)$. 
It follows $(\tilde{g}, y, x, y) \in \overline{\Delta} (R) = \Delta (R)$ by the assumption. 
Hence $\tilde{g} \in G (R)$ and 
$\tilde{g} \cdot y = x$, and hence 
$\phi := (\tilde{g}, y) \in (G \times_k W) (R)$ satisfies $\phi_K = h$ and $\pi \circ \phi = f$. 
Therefore by the valuative criterion of properness, $\pi$ is proper. 
\end{proof}

{\bf Step 2.} Let $W^0$ be the subset of $(C, L) \in W$ such that $C$ is smooth. 
Then $W^0$ is invariant under the action of $G$, and hence we can define a morphism 
$\pi^{0} := \pi |_{G \times_k W^0} : G \times _k W^0 \to W^0 \times_k W^0$. 
Let $\Delta^0$ be the graph of the morphism 
$\pi^0$. 
Then it is an open dense subset of $\Delta$, hence of $\overline{\Delta}$. 

\begin{Lem} \label{extension lemma}
For any DVR $R$, a morphism $f : {\rm Spec} \ R \to W \times_k W$ and a morphism 
$h_K : {\rm Spec} \ K \to G \times_k W^0$ such that $\pi \circ h_K = f_K$, 
there exists $\phi : {\rm Spec} \ R \to G \times_k W$ such that $\phi_K = h_K$ and 
$\pi \circ \phi = f$. 
\end{Lem}

\begin{proof}
Let $p_i$ be the $i$-th projection of $W \times _k W$. 
Let $({\mathcal C}, {\mathcal L})$ be the universal pair of cubics and lines over $W$. 
Let $S = {\rm Spec } \ R$, and we define 
\begin{gather*} 
X = (p_1 \circ f)^* ({\mathcal C}), \ 
Y = (p_2 \circ f)^* ({\mathcal C}),  \\
C = (p_1 \circ f)^* ({\mathcal C} \cap {\mathcal L}), \ 
D = (p_2 \circ f)^* ({\mathcal C} \cap {\mathcal L}).
\end{gather*}
Then the family of cubic curves $p : X \to S$ (resp. $q : Y \to S$) 
is proper flat over $S$ with an effective divisor $C$ (resp. $D$) of degree three 
flat over $S$. 
Note that $X \subset {\mathbb P}^2_R$ and $Y \subset {\mathbb P}^2_R$. 
Since $\pi \circ h_K = f_K$, there exists $g \in G (K)$ such that 
\begin{align*}
X_K = g (Y_K) , \ \ C_K = g (D_K). 
\end{align*} 
In particular, 
we have an isomorphism as pairs of $K$-schemes: 
\begin{align*}
\gamma ' : (Y_K, D_K) \longrightarrow (X_K, C_K). 
\end{align*}

To prove the existence of $\phi$, 
it suffices to show that 
there exists an isomorphism $\gamma : (Y, D) \to (X, C)$ as pairs of $R$-schemes such that $\gamma _K = \gamma'$. 
In fact, 
since $\gamma$ is induced from $G (R)$, 
there exists $\tilde{g} \in G(R)$ such that $\tilde{g}_K = g$, $X = \tilde{g} (Y)$ and 
$C = \tilde{g} (D)$. 
Then $\phi := (\tilde{g}, p_2 \circ f) \in (G \times_k W) (R)$ satisfies desired properties. 
In the following, we construct such isomorphism $\gamma$. 

Since $p_1 \circ f_K ({\rm Spec} \ K)
= p_1 \left( \pi \circ h_K ({\rm Spec} \ K) \right) \subset W^0$, 
$X_K$ is smooth over $K$, 
that is, $X_K$ is a family of smooth cubic curves over ${\rm Spec} \ K$, 
hence so is $Y_K$. 
Therefore by the minimal models theorem 
(see \cite{Lichtenbaum68} or \cite{Shafarevitch66}), 
there exists the minimal proper regular model $X^{\sharp}$ (resp. $Y^{\sharp}$) 
for $X_K$ (resp. $Y_K$). 
In fact, $X^{\sharp}$ (resp. $Y^{\sharp}$) 
is obtained by resolving the singularities of $X$ (resp. $Y$). 
Let $\nu : X^{\sharp} \to X$ (resp. $\mu : Y^{\sharp} \to Y$) be the minimal resolution. 
By the uniqueness of minimal models, 
the isomorphism $\gamma'$ extends to an isomorphism 
\begin{align*}
\gamma^{\sharp} : Y^{\sharp} \longrightarrow X^{\sharp}
\end{align*}
as $R$-schemes. 
Let $C^{\sharp}$ (resp. $D^{\sharp}$) be the proper transform of $C$ (resp. $D$) 
by $\nu$ (resp. $\mu$). 
Since $(\gamma ')^* (C_K) = D_K$ and $\gamma^{\sharp}_K = \gamma '$, 
we have $(\gamma ^{\sharp})^* (C^{\sharp}) = D^{\sharp}$, which is the closure of $D_K$. 
Thus $\gamma ^{\sharp}$ is an isomorphism of pairs 
\begin{align*}
\gamma ^{\sharp} : (Y^{\sharp}, D^{\sharp}) \longrightarrow (X^{\sharp}, C^{\sharp}). 
\end{align*}
Recall that the singularities of $X$ (resp. $Y$) are not contained in $C$ (resp. $D$). Hence 
we obtain $C^{\sharp} \simeq C$ and $D^{\sharp} \simeq D$. 
Moreover $C$ (resp. $D$) is relatively very ample on $X$ (resp. $Y$). 
Therefore the contraction $\nu : X^{\sharp} \to X$ (resp. $\mu : Y^{\sharp} \to Y$) 
is induced from the morphism 
\begin{gather*}
N : X^{\sharp} \longrightarrow 
{\mathbb P} \left( (p^{\sharp})_* (O _{X^{\sharp}} (C^{\sharp}) ) \right) 
\simeq {\mathbb P}^2_R \\ 
({\rm resp.} \ 
M : Y^{\sharp} \longrightarrow 
{\mathbb P} \left( (q^{\sharp})_* (O _{Y^{\sharp}} (D^{\sharp}) ) \right) 
\simeq {\mathbb P}^2_R ), 
\end{gather*}
where $p^{\sharp} := p \circ \nu : X^{\sharp} \to S$ (resp. 
$q^{\sharp} := q \circ \mu : Y^{\sharp} \to S$). 
Then we have $X = N (X^{\sharp})$ and $Y = M (Y^{\sharp})$, and 
the isomorphism $\gamma ^{\sharp} : (Y^{\sharp}, D^{\sharp}) \rightarrow (X^{\sharp}, C^{\sharp})$ 
induces an isomorphism 
\begin{align*}
(\gamma^{\sharp})^{*} : (p^{\sharp})_* (O _{X^{\sharp}} (C^{\sharp}) ) \longrightarrow 
(q^{\sharp})_* (O _{Y^{\sharp}} (D^{\sharp}) ), 
\end{align*}
and a commutative diagram 
\[\xymatrix{
Y^{\sharp} \ar[d]_{\simeq}^{\gamma^{\sharp}} \ar[r]^M 
& Y \ar[d]^{\gamma}_{\simeq} \ar@{}[r]|\subset & {\mathbb P}^2_R 
\ar[d]_{\simeq}^{{\mathbb P} ((\gamma^{\sharp})^*)}  \\
X^{\sharp} \ar[r]_{N} & X \ar@{}[r]|\subset & {\mathbb P}^2_R
}\]
Since $M (D^{\sharp}) = D$ and $N (C^{\sharp}) = C$, 
we have an isomorphism $\gamma : (Y, D) \to (X, C)$ as pairs of $R$-schemes such that 
$\gamma_K = \gamma^{\sharp}_K = \gamma'$. 
\end{proof}

{\bf Step 3.} We shall prove that $\pi$ is proper.
Suppose that $\pi$ is not proper. 
By Lemma \ref{cha. of properness}, 
there exists a point $x \in \overline{\Delta} (k) \setminus \Delta (k)$. 
Since $\Delta^0$ is open dense in $\overline{\Delta}$, we can choose a CDVR 
$R$ and a morphism $\alpha : {\rm Spec} \ R \to \overline{\Delta}$ 
such that 
\begin{align*}
\alpha ({\rm Spec} \ k (0)) = x \ \ \mbox{and} \ \ \alpha ({\rm Spec} \ K) \subset \Delta^0. 
\end{align*}
Let 
$f := p_{3, 4} \circ \alpha$ and $h := p_{1, 2} \circ \alpha$. 
Since $\alpha ({\rm Spec} \ K) \subset \Delta^0$, 
we have $h_K : {\rm Spec} \ K \to G \times_k W^0$ and 
$f_K = \pi \circ h_K$. 
By Lemma \ref{extension lemma}, 
there exists a morphism $\phi : {\rm Spec} \ R \to G \times _k W$ such that 
$\phi_K = h_K$ and $\pi \circ \phi = f$. 
This contradicts $\alpha ({\rm Spec} \ k (0)) = x \in \overline{\Delta} (k) \setminus \Delta (k)$ as we saw in the proof of only if part of Lemma \ref{cha. of properness}. 
It follows that $\pi$ is proper. 

\subsection{Proof of (2) in Lemma \ref{key of BP13}}

Let $(C, L)$ be any pair in $W$. 
Then the cubic curve $C$ is 
\begin{enumerate}[(i)]
\item an elliptic curve, 
\item a nodal curve, 
\item a $3$-gon, or 
\item an irreducible conic plus a line which does not tangent to the conic. 
\end{enumerate}
When (i), we may assume that $C = V(x_0^3 + x_1^3 + x_2^3 - 3 \mu x_0 x_1 x_2)$ for some 
$\mu \in k$ with $\mu^3 \neq 1$. 
The stabilizer ${\rm Stab}_{G} (C, L)$ is contained in the Hesse group 
$G_{216}$, which is of order $216$. 
When (ii), we may assume that $C = V(x_0^3 + x_1^3 - 3 x_0 x_1 x_2)$. 
Then ${\rm Stab}_{G} (C, L)$ is contained in ${\rm Stab}_{G} (C)$, 
and ${\rm Stab}_{G} (C)$ is generated by 
\begin{align*}
g_1 (\zeta_3) := \left[ 
\begin{array}{ccc}
1 & 0 & 0 \\
0 & \zeta_3^2 & 0 \\
0 & 0 & \zeta_3
\end{array}
\right], 
g_2 := \left[ 
\begin{array}{ccc}
0 & 1 & 0 \\
1 & 0 & 0 \\
0 & 0 & 1
\end{array}
\right], 
\end{align*}
where $\zeta_3$ is a primitive third root of unity. 
When (iii), since $L$ does not pass through singularities of $C$, 
we may assume that 
$(C, L) = (V (x_0 x_1 x_2), V (x_0 + x_1 + x_2) )$. 
Hence ${\rm Stab}_{G} (C, L)$ is permutations of coordinates $x_0$, $x_1$ and $x_2$. 
When (iv), 
since $L$ does not pass through singularities of $C$, 
we may assume that 
$(C, L) = (V (x_2 (x_2^2 + x_0 x_1)), V (x_0 + x_1 + a x_2) )$ for some $a \in k$. 
Then ${\rm Stab}_{G} (C)$ is generated by $g_1(\alpha)$ ($\alpha \in k^{\times}$) and $g_2$. 
Hence if $a \not = 0$, then ${\rm Stab}_{G} (C, L)$ is generated by $g_2$, and 
if $a = 0$, then ${\rm Stab}_{G} (C, L)$ is generated by $g_1 (-1)$ and $g_2$. 

In any case, ${\rm Stab}_{G} (C, L)$ is finite. 

\section{A comparison of $\overline{P}_{1,3}$ with ${BP}_{1,3}$} \label{comparison}

We give a birational map from $\overline{P}_{1,3}$ to $BP_{1,3}$. 
Let $W_{\rm T} \subset BP_{1, 3}$ be the subset of pairs $(C, L)$ consisting of a cubic curve $C$ and a line $L$ which is $3$-tangent to $C$, 
and let $W^{(sm)}_{\rm T}$ (resp. $W^{(node)}_{\rm T}$) be the subset 
$(C, L) \in W_T$ such that $C$ is smooth (resp. nodal). 
We show that $W_{\rm T}$ is isomorphic to ${\mathbb P}^1$ as follows: 
We have 
$
W^{(sm)}_{\rm T} = \left\{ z(\mu) \mid \mu \in k , \ \mu^3 \not = 1 \right\}/ G_{216}, 
$
where $G_{216}$ is the Hesse group and 
\begin{align*}
z(\mu) := \left( V (x_0^3 + x_1^3 + x_2^3 - 3 \mu x_0x_1x_2), V (\mu x_0 + x_1 + x_2) \right). 
\end{align*}
Thus $W^{(sm)}_{\rm T}$ is isomorphic to $k$ via the $j$-invariant, that is, 
$z(\mu) \mapsto j (\mu) = \frac{\mu^3 (\mu^3 + 8)^3}{(\mu^3 - 1)^3}$. 
On the other hand, since $\mu \not = 1$, in $BP_{1,3}$ we have 
\begin{align*}
z (\mu) 
&= z \left( \frac{\mu+2}{\mu-1} \right) \\
&= 
\left( V \left( \left(\frac{\mu-1}{\mu+2} \right)^3 x_0^3 + x_1^3 + x_2^3 - 3 x_0x_1x_2 \right), 
V (x_0 + x_1 + x_2) \right), 
\end{align*} 
and hence 
$\displaystyle \lim_{\mu\to 1} z (\mu) = \left( V (x_1^3 + x_2^3 - 3 x_0x_1x_2), V (x_0 + x_1 + x_2) \right) \in W^{(node)}_{\rm T}$. 
Since $W^{(node)}_{\rm T}$ is the single point, 
$W_{\rm T} = W^{(sm)}_{\rm T} \cup W^{(node)}_{\rm T} \simeq \mathbb{P}^1$. 

Recall that $W_{\rm Cusp} \subset \overline{P}_{1,3}$ is the subset of semistable pairs $(C, L)$ consisting of a cuspidal cubic $C$ and a line $L$ intersecting at smooth points of $C$, which is not $3$-tangent to $C$. 
By Lemma \ref{Cusp}, $W_{\rm Cusp}$ is isomorphic to ${\mathbb P}^1$. 
We denote by $Y$ (resp. $Z$) the open dense subset $\overline{P}_{1, 3} \setminus W_{\rm Cusp}$ (resp. $BP_{1, 3} \setminus W_{\rm T}$). 
Then the identity map $f : Y \rightarrow Z$, $(C, L) \mapsto (C, L)$ defines a rational map 
$f : \overline{P}_{1, 3} \rightarrow BP_{1, 3}$ 
and the inverse map $f^{-1} : Z \rightarrow Y$ defines the inverse rational map of $f$. Hence 
$f : \overline{P}_{1, 3} \rightarrow BP_{1, 3}$ is a birational map, and the base loci of $f$, $f^{-1}$ are $W_{\rm Cusp}$, $W_{\rm T}$, respectively. 
Let $G(f)$ be the graph of $f : Y \to Z$: 
$$
G(f) := {\rm Im} \left( ({\rm id}, f) : Y \longrightarrow Y \times_k Z \right) \subset \overline{P}_{1, 3} \times_k BP_{1, 3}.  
$$

\begin{Lem}  \label{graph}
Let $\Gamma := W_{\rm Cusp} \times_k W_{\rm T} \simeq  {\mathbb P}^1 \times_k {\mathbb P}^1$ and $p_i$ be the $i$-th projection of $\Gamma$. 
Then for any $x \in \Gamma$, there exists a semistable pair $({\mathcal C}_{t}, {\mathcal L}_{t}) \in Y$ such that 
\begin{align*}
\underset{t \rightarrow 0} \lim \ ({\mathcal C}_{t}, {\mathcal L}_t) = p_1 (x) \ \ {\rm and} \ \ \underset{t \rightarrow 0} \lim \ f ({\mathcal C}_{t}, {\mathcal L}_t) = p_2 (x),
\end{align*}
In particular, we have $\Gamma \subset \overline{G(f)} \setminus G(f)$. 
\end{Lem}

\begin{proof}
For any $x \in \Gamma$, we have $p_1(x) \in W_{\rm Cusp}$ and $p_2(x) \in W_{\rm T}$. 
Let $p_1(x)= (C, L)$ and $p_2(x)= (C', L')$. 
Since $C$ is a cuspidal curve, $L$ does not pass through the cusp point of $C$, and $L$ is not $3$-tangent to $C$, 
we may assume that 
\begin{align*}
C : x_0 x_2^2 = x_1^3, \ \ L : x_0 = b_1 x_1 + b_2 x_2 
\end{align*}
with $(b_1, b_2) \not = (0, 0)$. 
Since $C'$ is a smooth or nodal curve, and $L'$ is a $3$-tangent at a smooth point of $C'$, 
we may assume that 
\begin{align} \label{C'L'}
C' : x_0 x_2^2 = x_1^3 + B_1 x_0^2 x_1 + B_2 x_0^3, \ \ L' : x_0 = 0 
\end{align}
with $(B_1, B_2) \not = (0, 0)$. 
In fact, assume $L' : x_0 = 0$. Then $C' \cap L'$ is a triple point, hence $C'$ has the form 
$$
C' : a_{02} x_0 x_2^2 + a_{11} x_0 x_1 x_2 + a_{01} x_0^2 x_2 = x_1^3 + a_{20} x_0 x_1^2 + a_{10} x_0^2 x_1 + a_{00} x_0^3 .
$$
If $a_{02} = 0$, then $C'$ has a singular point $(0 : 0 : 1)$ and $L'$ passes through this point, which is absurd. 
Hence $a_{02} \not = 0$. We may assume that $a_{02} = 1$. Thus $C'$ has a Weierstrass form. 
Since characteristic of $k$ is not equal to $2$ and $3$, we may assume that $(C', L')$ has the above form (\ref{C'L'}). 

For $t \not = 0$, we define a pair $({\mathcal C}_t, {\mathcal L}_t)$ as follows: 
\begin{align*}
{\mathcal C}_t : x_0 x_2^2 = x_1^3 + B_1  t^4 x_0^2 x_1 + B_2 t^6 x_0^3, \ \ {\mathcal L}_t : x_0 = b_1 x_1 + b_2 x_2. 
\end{align*}
Let $g = {\rm Diag} \ (1, 1/t^2, 1/t^3) \in {\rm PGL}(3)$.  
Then we obtain 
\begin{align*}
g \cdot {\mathcal C}_t : x_0 x_2^2 = x_1^3 + B_1 x_0^2 x_1 + B_2 x_0^3, \ \ g \cdot {\mathcal L}_t : x_0 = b_1 t^2 x_1 + b_2 t^3 x_2. 
\end{align*}
Therefore we have 
$\underset{t \rightarrow 0} \lim \ ({\mathcal C}_{t}, {\mathcal L}_t) = (C, L) = p_1 (x)$ and 
\begin{align*}
\underset{t \rightarrow 0} \lim \ f({\mathcal C}_{t}, {\mathcal L}_t) = \underset{t \rightarrow 0} \lim \ ({\mathcal C}_{t}, {\mathcal L}_t) = \underset{t \rightarrow 0} \lim \ (g\cdot {\mathcal C}_{t}, g \cdot {\mathcal L}_t) = (C', L') = p_2 (x). 
\end{align*}
\end{proof}

\begin{Th} \label{main1}
Let $p_i$ be the $i$-th projection of $\overline{P}_{1, 3} \times_k BP_{1, 3}$. 
Then $f : \overline{P}_{1, 3} \rightarrow BP_{1, 3}$ is a birational map and 
\begin{align*}
W_{\rm Cusp} \times_k W_{\rm T} = p_1^{-1} \left( W_{\rm Cusp} \right) \cap p_2^{-1} \left( W_{\rm T} \right) = \overline{G(f)} \setminus G(f). 
\end{align*}
\end{Th}

\begin{proof}
Let $\Gamma ^\dagger = p_1^{-1} \left( W_{\rm Cusp} \right) \cap p_2^{-1} \left( W_{\rm T} \right)$. 
Clearly, $\Gamma = W_{\rm Cusp} \times_k W_{\rm T} \subset \Gamma ^ \dagger$. 
On the other hand, by Lemma \ref{graph}, $\Gamma \subset \overline{G(f)} \setminus G(f)$. 
Hence it suffices to show that $\Gamma ^\dagger  \subset \Gamma$ and $\overline{G(f)} \setminus G(f) \subset \Gamma ^\dagger $. 

We first show that $\Gamma ^\dagger  \subset \Gamma$. 
We denote by $\iota $ (resp. $\iota ^\dagger $) the canonical inclusion from $\Gamma$ (resp. $\Gamma ^\dagger $) to $\overline{P}_{1, 3} \times {BP}_{1, 3}$. 
Since $\Gamma ^{\dagger} \subset \overline{P}_{1, 3} \times {BP}_{1, 3}$, 
$p_1 \mid_{\Gamma^{\dagger}} : \Gamma^{\dagger} \to W_{\rm Cusp}$ and 
$p_2 \mid_{\Gamma^{\dagger}} : \Gamma^{\dagger} \to W_{\rm T}$ make a commutative diagram with the morphisms 
$W_{\rm Cusp} \to {\rm Spec} \ k$ and 
$W_{\rm T} \to {\rm Spec} \ k$. 
By the definition of the fibered product $\Gamma = W_{\rm Cusp} \times_k W_{\rm T}$, 
there exists a unique morphism $\theta  : \Gamma ^\dagger \rightarrow \Gamma$ such that $p_i \mid _{\Gamma} \circ \ \theta = p_i \mid _{\Gamma ^\dagger}$ ($i = 1, 2$). 
Then $\iota^\dagger $ and $\iota \circ \theta $ are morphisms from 
$\Gamma ^\dagger $ to $\overline{P}_{1, 3} \times {BP}_{1, 3}$ which satisfy 
$
p_i \circ (\iota \circ \theta ) = p_i \mid _{\Gamma} \circ \ \theta 
= p_i \mid _{\Gamma ^\dagger }
= p_i \circ \iota^\dagger
$ ($i = 1$, $2$). 
By the definition of the fibered product $\overline{P}_{1, 3} \times {BP}_{1, 3}$, such morphism is unique. Hence $\iota^\dagger = \iota \circ \theta $, and hence $\Gamma ^\dagger \subset \Gamma$. 

Next we show that $\overline{G(f)} \setminus G(f) \subset \Gamma ^\dagger $. 
For any $x \in \overline{G(f)} \setminus G(f)$, there exists $x_t \in G(f)$ such that $\underset{t \rightarrow 0}{\lim} \ x_t = x$. 
Suppose $p_1 (x) \in Y = \overline{P}_{1, 3} \setminus W_{\rm Cusp}$. 
Then 
\begin{eqnarray*}
Z \ni f (p_1 (x)) = \underset{t \rightarrow 0}{\lim} \ f (p_1 (x_t)) = \underset{t \rightarrow 0}{\lim} \ p_2 (x_t) = p_2 (x) .
\end{eqnarray*}
Thus $x \in G(f)$, which is absurd. Hence $p_1 (x) \in W_{\rm Cusp}$. 
Since $f$ is birational, we get $p_2 (x) \in W_{\rm T}$ similarly. 
Thus we obtain $\overline{G(f)} \setminus G(f) 
\subset p_1^{-1} \left( W_{\rm Cusp} \right) \cap p_2^{-1} \left( W_{\rm T} \right) 
= \Gamma ^\dagger $.
\end{proof}

\section{A comparison of $\overline{P}_{1, 3}$ and ${BP}_{1, 3}$ with 
$SQ_{1,3} \times {\mathbb P}^2$}

In this section, in order to compare $\overline{P}_{1,3}$ and ${BP}_{1,3}$ with $SQ_{1,3} \times {\mathbb P} (V)$, 
we construct natural morphisms from blowing-ups of $SQ_{1, 3} \times {\mathbb P} (V)$ to ${BP}_{1, 3}$ and $\overline{P}_{1, 3}$, respectively, 
where ${\mathbb P} (V) \simeq {\mathbb P}^2$ is the space of lines on ${\mathbb P}^2$ 
(see Section \ref{GIT moduli}) and 
$SQ_{1, 3} \simeq {\mathbb P}^1$ is the moduli of Hesse cubics defined in \cite{Nakamura99}, which is well-known classically as the modular curve of level three. 
In addition, we will see a relation between the birational map $f$ in Theorem \ref{main1} and these two morphisms. 
Note that the universal Hesse cubic over $SQ_{1,3}$ is given by 
\begin{align*}
\mu_0 (x_0^3 + x_1^3 + x_2^3 ) = 3 \mu_1 x_0 x_1 x_2 , \ \ \ (\mu_0 : \mu_1) \in {\mathbb P}^1. 
\end{align*}

\begin{Rem}
In this paper, we consider the moduli $SQ_{1,3}$ over an algebraically closed field $k$ with 
${\rm ch} (k) \not = 2$, $3$. 
However in \cite{Nakamura99}, I. Nakamura considers the moduli spaces over $\Z [\zeta_3, 1/3]$, 
where $\zeta_3$ is a primitive third root of unity. 
\end{Rem}

Let $X = SQ_{1,3} \times {\mathbb P} (V) \simeq {\mathbb P}^1 \times {\mathbb P}^2$. 
We define rational maps $\varphi : X \to {BP}_{1,3}$ and 
$\psi : X \to \overline{P}_{1,3}$ forgetting the level structure, that is, 
\begin{gather} 
\varphi \ ({\rm resp.} \ \psi) : (\mu_0 : \mu_1) \times (b_0 : b_1 : b_2) \longmapsto (C, L), 
\label{phi and psi} \\ 
C : \mu_0 (x_0^3 + x_1^3 + x_2^3) = 3 \mu_1 x_0 x_1 x_2, \ \ L : b_0 x_0 + b_1 x_1 + b_2 x_2 = 0 . 
\nonumber 
\end{gather}

\begin{Lem} \label{natural maps}
For each $\ell \in \mathbb{Z}$, let $[\ell] = \ell \mod 3 \in \{ 0, 1, 2 \}$. 
Then we have following properties. 
\begin{enumerate}[(i)]
\item 
The base locus of $\varphi$ is a union of four $3$-gons 
\begin{align*}
B^ {(a)} = \left\{ (a : 1 ) \times (b_0 : b_1 : b_ 2) \in X 
\; \middle| \; b^{(a)}_0 b^{(a)}_1 b^{(a)}_2 = 0 
\right\}, 
\end{align*}
where $a=0$ or $a ^3 = 1$, and 
$b_i^{(a)} = \left\{
\begin{array}{l l} 
b_i & (a=0), \\
b_0 + \zeta _3^i b_1 + a \zeta _3^{2 i} b_2 & (a^3=1). \\
\end{array} \right.$
\item 
The base locus of $\psi$ is a union of above four $3$-gons $B^{(a)}$ and nine lines 
\begin{align*}
A^{(j)}_i =  \left\{ 
(\mu_0 : \mu_1) \times (b_0 : b_1 : b_ 2) \in X \; \middle| \; 
\begin{array}{l}
b_{[i+1]} = b_{[i+2]} \zeta_3^{2 j}, \\ 
b_{[i+2]} \mu_1 = b_{i} \zeta_3^{2 j} \mu_0 \\
\end{array}
\right\}, 
\end{align*}
where $i$,$j \in \{0, 1, 2 \}$. 
\item 
Let $O^{(a,i)}$ be the vertex 
$V \left( \mu_0 - a \mu_1, b^{(a)}_{[i+1]}, b^{(a)}_{[i+2]} \right)$ of 
$B^{(a)}$ ($a=0$ or $a^3=1$, and $0 \leq i \leq 2$). 
Then each line $A^{(j)}_i$ passes through only four vertexes 
$O^{(0,i)}$ and $O^{\left( \zeta_3^k, [2 j + (2-i)k] \right)}$ ($0 \leq k \leq 2$), and 
each vertex $O^{(0, i)}$ (resp. $O^{\left( \zeta_3^k, i \right)}$) is contained in only 
three lines $A_i^{(j)}$ (resp. $A_j^{([2k (j+1) + 2i])}$) ($0 \leq j \leq 2$). 
\item 
Let $B = \bigcup B^{(a)}$ and $A = \bigcup A^{(j)}_i$. 
Let $f$ be the birational map in Theorem \ref{main1}. 
Then we have $\varphi = f \circ \psi$ on $X \setminus \left( B \cup A  \right)$. 
\item 
$\varphi$ and $\psi$ are generically finite of degree $216$. 
\end{enumerate}
\end{Lem}

\begin{proof}
(iii) and (iv) are clear from definitions. We prove (i), (ii) and (v). 

Firstly, we prove (i). 
$\varphi$ is not well-defined if 
$L$ passes through a singular point of $C$. 
A Hesse cubic $C$ has singular points if and only if $\mu_0 / \mu_1 = 0$, $1$, $\zeta_3$ or 
$\zeta_3^2$. Then $C$ is a $3$-gon and its singularities are 
\begin{gather*}
 (1 :  0 : 0), \ (0 : 1 : 0), \ (0 : 0 : 1) \ \ \mbox{if $\mu_0/ \mu _1 = 0$, and} \\ 
(1 :  1 : a), \ 
(1 : \zeta_3 : a \zeta_3^{2}), \ (1 : \zeta_3^{2} : a \zeta_3) \ \ 
\mbox{if $\mu_0/ \mu _1 = a$ ($a^3=1$)}. 
\end{gather*}
Hence a line $L : b_0 x_0 + b_1 x_1 + b_2 x_2 = 0$ passes through a singular point of $C$ 
if and only if 
$\mu_0 / \mu_1 = a$ and $b^{(a)}_0 b^{(a)}_1 b^{(a)}_2 = 0$, where $a = 0$ or $a^3 = 1$. 
Therefore the base locus of $\varphi$ is the union of four $3$-gons $B^{(a)}$ 
($a=0$ or $a^3=1$). 

Secondly, we prove (ii). 
$\psi$ is not well-defined if $L$ passes through a singular point of $C$, or $L$ is $3$-tangent to $C$. 
If $L$ is $3$-tangent to $C$, then $L$ is tangent to $C$ at one of the nine inflection points 
\begin{align*}
(0 : 1 : -\zeta_3^i), \ 
(1 : 0 : -\zeta_3^i), \ 
(1 : -\zeta_3^i : 0) \ \ (0 \leq i \leq 2)
\end{align*}
of $C$. Hence we can check easily that 
$L$ is $3$-tangent to $C$ 
if and only if 
$b_{[i+1]} = b_{[i+2]} \zeta_3^{2 j}$ and $b_{[i+2]} \mu_1 = b_i \zeta_3^{2 j} \mu_0$ 
for some $i$, $j \in \{0, 1, 2\}$. 
Therefore the base locus of $\psi$ is the union of four $3$-gons $B^{(a)}$ ($a=0$ or $a^3=1$) 
and nine lines $A^{(j)}_i$ ($i$, $j \in \{0, 1, 2\}$). 

Finally, we prove (v). 
The Hesse group $G_{216}$ is the subgroup of ${\rm PGL} (3)$ consisting of 
$g \in {\rm PGL} (3)$ such that $g \cdot C$ is a Hesse cubic.  
It is generated by 
\begin{align*}
\sigma = \begin{bmatrix}
1 & 0 & 0 \\
0 & \zeta_3 & 0 \\
0 & 0 & \zeta_3^2 \\
\end{bmatrix}, \ 
\tau = \begin{bmatrix}
0 & 0 & 1 \\
1 & 0 & 0 \\
0 & 1 & 0 \\
\end{bmatrix}, \
\sigma_1 = \begin{bmatrix}
1 & 1 & 1 \\
1 & \zeta_3 & \zeta_3^2 \\
1 & \zeta_3^2 & \zeta_3 \\
\end{bmatrix}, \ 
\sigma_2 = \begin{bmatrix}
\zeta_3 & 0 & 0 \\
0 & 0 & 1 \\
0 & 1 & 0 \\
\end{bmatrix}
\end{align*}
and its order is $216$ (see Theorem 3.1.7 in \cite{Dolgachev12}). 
Then the action of ${\rm PGL} (3)$ on $\mathbb{P} (S^3 V)$ defined in Section 2 induces an 
action of $G_{216}$ on $SQ_{1,3}$ given  by 
\begin{gather*}
\sigma \cdot (\mu_0 : \mu_1) = (\mu_0 : \mu_1), \ \ 
\tau \cdot (\mu_0 : \mu_1) = (\mu_0 : \mu_1), \\
\sigma_1 \cdot (\mu_0 : \mu_1) = (\mu_1 - \mu_0 : \mu_1 + 2 \mu_0 ), \ \ 
\sigma_2 \cdot (\mu_0 : \mu_1) = (\mu_0 : \zeta_3^2 \mu_1). 
\end{gather*}
Then for any point $\mu = (\mu_0 : \mu_1) \in SQ_{1,3}$, 
the orbit consists of the following $12$ points: 
\begin{align*}
(\mu_0 : \alpha \mu_1), \ (\beta \mu_1 - \mu_0 : \gamma (\beta \mu_1 + 2 \mu_0 ) ) \ \ (\alpha ^3 = \beta ^3 = \gamma ^3 = 1). 
\end{align*}
On the other hand, 
the stabilizer subgroup $G_{\mu} \subset G_{216}$ is generated by 
\begin{align*}
\sigma = \begin{bmatrix}
1 & 0 & 0 \\
0 & \zeta_3 & 0 \\
0 & 0 & \zeta_3^2 \\
\end{bmatrix}, \ 
\tau = \begin{bmatrix}
0 & 0 & 1 \\
1 & 0 & 0 \\
0 & 1 & 0 \\
\end{bmatrix}, \
\sigma_3 := \sigma_1^2 =
\sigma_2^3 = \begin{bmatrix}
1 & 0 & 0 \\
0 & 0 & 1 \\
0 & 1 & 0 \\
\end{bmatrix}
\end{align*}
and its order is $18$. Moreover for any point $b = (b_0 : b_1 : b_2) \in {\mathbb P} (V)$, 
the orbit $G_{\mu} \cdot b$ consists of the following $18$ points: 
\begin{align*}
(b_i : \delta b_j : \delta^2 b_k) \ \ (\delta^3 = 1 , \ \{ i, j, k \} = \{ 0, 1, 2 \}). 
\end{align*}
Hence in general, for any pair $(C, L)$ given by (\ref{phi and psi}), 
$\varphi ^{-1} (C, L)$ and $\psi^{-1} (C,L)$ consist of the following $216$ points, respectively: 
\begin{align*}
(\mu_0 : \alpha \mu_1) 
\times (b_i : \delta b_j : \delta^2 b_k), 
 \ (\beta \mu_1 - \mu_0 : \gamma (\beta \mu_1 + 2 \mu_0 ) ) 
\times (b_i : \delta b_j : \delta^2 b_k) , 
\end{align*} 
where $\alpha ^3 = \beta^3 = \gamma^3 = \delta^3 = 1$ and $\{ i, j, k \} = \{ 0, 1, 2 \}$. 
Therefore $\varphi$ and $\psi$ are generically finite of degree $216$. 
\end{proof}

In the rest of this section, 
we will construct explicitly morphisms $\pi : \tilde{X} \to X$ and $p : \hat{X} \to \tilde{X}$ 
as compositions of blowing-ups, 
and construct morphisms $\tilde{\varphi} : \tilde{X} \to BP_{1,3}$ and 
$\hat{\psi} : \hat{X} \to \overline{P}_{1,3}$ such that 
$\tilde{\varphi}$ and $\hat{\psi}$ are extensions of $\varphi$ and $\psi$ respectively.  
Meanwhile, we will prove the following Proposition: 
\begin{Prop} \label{main2}
We have a commutative diagram 
\begin{eqnarray*}
\begin{CD}
\hat{X} @> \hat{\psi} >> \overline{P}_{1, 3}\\
@V p  VV  @VV f V\\
\tilde{X} @> \tilde{\varphi} >> {BP}_{1, 3}\\
@V\pi  VV \\
X
\end{CD}
\end{eqnarray*}
where $X = SQ_{1,3} \times {\mathbb P}(V) \simeq {\mathbb P}^1 \times {\mathbb P}^2$, 
$f$ is the birational in Theorem \ref{main1}, 
and $\pi$ and $p$ are compositions of blowing-ups 
with nonsingular (reducible) centers, respectively. 
Then we have the following properties: 
\begin{enumerate}[(i)]
\item Morphisms $\tilde{\varphi}$ and $\hat{\psi}$ are extensions of $\varphi$ and $\psi$ respectively. In particular,  they are generically finite of degree $216$. 
\item The base locus of $f$ (resp. $f^{-1}$) is $W_{\rm Cusp} \simeq {\mathbb P}^1$ (resp. 
$W_{\rm T} \simeq {\mathbb P}^1$), and 
\item $\tilde{\varphi}^{-1} (W_{\rm T}) = \tilde{A}$ and $\hat{\psi}^{-1} (W_{\rm Cusp}) = \hat{E}$, where $\tilde{A}$ (resp. $\hat{E}$) is the center (resp. the exceptional set) of 
$p : \hat{X} \to \tilde{X}$. 
\end{enumerate}
\end{Prop}

\subsection{Constructions of morphisms $\pi$ and $\tilde{\varphi}$} 
\label{to BP}

We construct $\pi : \tilde{X} \to X$ and $\tilde{\varphi} : \tilde{X} \to BP_{1,3}$ in Steps 1-4. 

{\bf Step 1.} 
We take an affine open covering of $X$ consisting of $15$ affine subsets $U^{(a, i)}$ 
($a = 0$, $\infty$ or $a^3=1$, and $0 \leq i \leq 2$) 
such that $\varphi$ is well-defined on each $U^{(\infty, i)}$, 
and the base locus of $\varphi$ on each $U^{(a, i)}$ ($a \neq \infty$) is a union of two axes. 
In fact, we define $U^{(a, i)}$ as follows: 
For $a = 0$, $\infty$ or $a^3=1$, we put 
\begin{align*}
\mu_0^{(a)} &= \left\{
\begin{array}{l l} 
\mu_0 & (a=0), \\
a \mu_1 - \mu_0 & (a^3=1), \\
\mu_1 & (a=\infty), \\
\end{array} \right. 
\ \ 
\mu_1^{(a)} = \left\{
\begin{array}{l l} 
\mu_1 & (a=0), \\
a \mu_1 + 2 \mu_0 & (a^3=1), \\
\mu_0 & (a=\infty), \\
\end{array} \right. 
\\
% \end{align*}
% \begin{align*}
b_i^{(a)} &= \left\{
\begin{array}{l l} 
b_i & (a=0, \infty), \\
b_0 + \zeta _3^i b_1 + a \zeta _3^{2 i} b_2 & (a^3=1), \\
\end{array} \right. \ \ (0 \leq i \leq  2)\\
u^{(a)} &= \mu_0^{(a)} / \mu_1^{(a)} , \ \
s^{(a)}_{i, j} = b^{(a)}_i / b^{(a)}_j . 
\end{align*}
Let 
\begin{equation*} 
U^{(a, i)} = {\rm Spec} \left( k \left[ u^{(a)}, s^{(a)}_{[i+1], i}, s^{(a)}_{[i+2], i} , \frac{1}{\left( u^{(a)} \right) ^3 - 1}\right] \right), 
\end{equation*}
where $[\ell] = \ell \mod 3 \in \{ 0, 1, 2 \}$. 
Note that when $a\ne\infty$, the origin of $U^{(a,i)}$ is $O^{(a,i)}$ 
in Lemma \ref{natural maps}, (iii). 
Then $X$ has an affine open covering 
\begin{align*}
X = \bigcup_{\substack{a = 0, \infty \ {\rm or} \ a^3 = 1, \\ 0 \leq i \leq 2}} U^{(a, i)} .
\end{align*}
By Lemma \ref{natural maps}, (i), 
$\varphi$ is well-defined on $U^{\infty} := \bigcup_{i=0}^2 U^{(\infty, i)}$ and the base locus of 
$\varphi$ on each $U^{(a, i)}$ ($a \neq \infty$) is 
the union of the $s^{(a)}_{[i+1], i}$-axis and the $s^{(a)}_{[i+2], i}$-axis, that is, 
the union of two lines $V \left( u^{(a)}, s^{(a)}_{[i+2], i} \right)$ and 
$V \left( u^{(a)}, s^{(a)}_{[i+1], i} \right)$. 

{\bf Step 2.} 
We blow up each $U^{(a, i)}$ ($a \ne \infty$) along the base locus of $\varphi$. 
We give a construction in detail in the case of $U^{(0, 0)}$, since other cases are similar. 
For simplicity, we put $(u, s_1, s_2) = \left( u^{(0)}, s^{(0)}_{1,0}, s^{(0)}_{2,0} \right)$ and 
$U = U^{(0, 0)}$. 
Let $L_j$ ($j=1$, $2$) be the $s_j$-axis, that is, 
$L_1 = V \left( u, s_2 \right)$ and $L_2 = V \left( u, s_1 \right)$. 
We blow up $U$ as follows: 
\begin{equation*} 
U \overset{\pi _1}{\longleftarrow } 
B_{Z_1} (U)  \overset{\pi _2}{\longleftarrow } 
B_{Z_2} B_{Z_1} (U)  \overset{\pi _3}{\longleftarrow } 
\tilde{U} := B_{Z_3} B_{Z_2} B_{Z_1} (U), 
\end{equation*}
where the symbol $B_Z (Y)$ means the blowing-up of $Y$ along $Z$. 
We denote by $\pi$ the composition of blowing-ups $\pi_1$, $\pi_2$ and $\pi_3$. 
Then centers $Z_1$, $Z_2$ and $Z_3$ are defined as follows: 
\begin{itemize}
\item $Z_1$ is the origin of $U$, that is, 
$Z_1 = L_1 \cap L_2 = O^{(0, 0)}$, 
\item $Z_2$ is a disjoint union of $2$ lines which are proper transforms $\tilde{L}_j$ ($j = 1$, $2$) of $L_j$ under the blowing-up $\pi_1$, and 
\item $Z_3$ is a disjoint union of $2$ lines $L'_j$ ($j=1$, $2$) in the exceptional set of $\pi_2$. 
Here $L'_1$ and $L'_2$ are defined as follows: 
$B_{Z_2} B_{Z_1} (U)$ has an affine open covering consisting of 
$5$ affine open subsets 
\begin{gather}
{\rm Spec} \left( k \left[ u, \frac{s_{1}}{u}, \frac{s_{2}}{u}, \frac{1}{ u^3 - 1} \right] \right) , 
\nonumber 
\ \  
{\rm Spec} \left( k \left[ s_{2}, \frac{u}{s_{2}}, \frac{s_{1}}{u}, 
\frac{1}{ s_2^3 \left(\frac{u}{s_2} \right)^3 - 1} \right] \right) , 
\nonumber 
\\ 
% \end{gather}
% \begin{gather}
{\rm Spec} \left( k \left[ s_{2}, \frac{u}{s_{1}}, \frac{s_{1}}{s_2}, 
\frac{1}{ s_2^3 \left(\frac{u}{s_1} \right)^3 \left( \frac{s_1}{s_2} \right) ^3 - 1} \right] \right) , 
\label{U2 (3)} 
%\\ 
 \end{gather}
 \begin{gather}
{\rm Spec} \left( k \left[ s_{1}, \frac{s_{2}}{s_1}, \frac{u}{s_{2}}, 
\frac{1}{ s_1^3 \left(\frac{s_2}{s_1} \right)^3 \left( \frac{u}{s_2} \right) ^3- 1} \right] \right) , 
\label{U2 (4)} 
\\ 
% \end{gather}
% \begin{gather}
{\rm Spec} \left( k \left[ s_{1}, \frac{s_{2}}{u}, \frac{u}{s_{1}}, 
\frac{1}{s_1^3 \left(\frac{u}{s_1} \right)^3 - 1} \right] \right) . \nonumber 
\end{gather}
Then we define $L'_1$ (resp. $L'_2$) as the $s_1$
(resp. $s_2$)-axis in 
(\ref{U2 (4)}) (resp. (\ref{U2 (3)})). 
\end{itemize}
Then $\tilde{U} = B_{Z_3} B_{Z_2} B_{Z_1} (U)$ has an affine open covering consisting of 
$7$ affine open subsets defined by 
\begin{gather*}
U_1 
= {\rm Spec} \left( 
k \left[ u, \frac{s_1}{u}, \frac{s_2}{u} , \frac{1}{u_1^3 - 1} \right] \right) , \ \ 
% \end{align}
% \begin{align}
U_2 
= {\rm Spec} \left( k \left[ s_2, \frac{u}{s_2}, \frac{s_1}{u} , 
\frac{1}{ u_2 ^3 - 1} \right] \right) , 
\\ 
% \end{gather*}
% \begin{gather*}
U_3 
= {\rm Spec} \left( k \left[ s_2, \frac{u}{s_1}, \frac{s_1^2}{u s_2}, 
\frac{1}{ u_3^3 - 1} \right] \right) , 
 \ \ 
%\end{align}
%\begin{align}
U_4 
= {\rm Spec} \left( k \left[ s_2, \frac{u s_2}{s_1^2}, \frac{s_1}{s_2} , 
\frac{1}{ u_4 ^3 - 1} \right] \right) , 
\\ 
% \end{align}
% \begin{align}
U_5 
= {\rm Spec} \left( k \left[ s_1, \frac{u s_1}{s_2^2}, \frac{s_2}{s_1} , 
\frac{1}{ u_5 ^3 - 1} \right] \right) , 
\ \ 
% \end{align}
% \begin{align}
U_6 
= {\rm Spec} \left( k \left[ s_1, \frac{u}{s_2}, \frac{s_2^2}{u s_1}, 
\frac{1}{ u_6^3 - 1} \right] \right) , 
% \\ 
 \end{gather*}
 \begin{gather*}
U_7 
= {\rm Spec} \left( k \left[ s_1, \frac{u}{s_1}, \frac{s_2}{u} , 
\frac{1}{ u_7 ^3 - 1} \right] \right) , 
\end{gather*}
where for each $j \in \{ 1, \dots, 7 \}$, we denote by $v^{(j)}_0$, $v^{(j)}_1$, $v^{(j)}_2$ the above coordinates of $U_j$, and $u_j$ is given by
\begin{equation} \label{u^(a,i)_j}
u_j = \left\{
\begin{array}{l l} 
v^{(1)}_0 & (j = 1), \\
v^{(j)}_0 v^{(j)}_1 & (j = 2, 7), \\
v^{(j)}_0 \left( v^{(j)}_1 \right)^2 v^{(j)}_2 & (j = 3, 6), \\
v^{(j)}_0 v^{(j)}_1 \left( v^{(j)}_2 \right)^2 & (j = 4, 5). \\
\end{array} \right. 
\end{equation}

{\bf Step 3.} 
We define a morphism $\tilde{\varphi}$ from $\tilde{U}$ to ${BP}_{1, 3}$ such that 
$\tilde{\varphi}$ is an extension of $\varphi$. 
For it, we define morphisms $\tilde{\varphi}_j$ from $U_j$ to ${BP}_{1, 3}$, respectively, 
and we glue these morphisms. 
Let 
\begin{gather*}
F_{(j)} (x) := u_j \left( x_0^3 + x_1^3 + x_2^3 \right) - 3 x_0 x_1 x_2, \ 
S_{(j)} (x) := x_0 + s^{(j)}_1 x_1 + s^{(j)}_2 x_2 
\end{gather*}
for each $j \in \{1, \dots, 7\}$, where 
\begin{align*}
s^{(j)}_1 := \left\{ 
\begin{array}{cl}
v^{(1)}_0 v^{(1)}_1 & (j = 1), \\
v^{(j)}_0 v^{(j)}_1 v^{(j)}_2 & (j = 2, 3), \\
v^{(4)}_0 v^{(4)}_2 & (j = 4), \\
v^{(j)}_0 & (j = 5, 6, 7), 
\end{array}
\right. \ 
s^{(j)}_2 := \left\{ 
\begin{array}{cl}
v^{(j)}_0 v^{(j)}_2 & (j = 1, 5), \\
v^{(j)}_0 & (j = 2, 3, 4), \\
v^{(j)}_0 v^{(j)}_1 v^{(j)}_2 & (j = 6, 7).
\end{array}
\right.
\end{align*}
Then we define $\tilde{\varphi}_j$ formally as follows: 
\begin{align*}
\tilde{\varphi}_j \left( v^{(j)}_0, v^{(j)}_1, v^{(j)}_2 \right) 
&:= M_j \cdot \left( V \left(F_{(j)} (x) \right), V \left( S_{(j)} (x) \right) \right) \\
&= \left( 
V \left( F_{(j)} \left( M_{j}^{-1} \cdot x \right) \right), 
V\left( S_{(j)} \left( M_{j}^{-1} \cdot x \right) \right) \right)
\ \ (1 \leq j \leq 7), 
\end{align*}
where $M_{j} := {\rm Diag} \left(1, m^{(j)}_1, m^{(j)}_2\right)$ and 
\begin{eqnarray*}
m^{(j)}_1 := \left\{ 
\begin{array}{cl}
u_1 & (j = 1), \\
\left( u_j s^{(j)}_2 \right)^{\frac{1}{2}} & (j = 2, 3), \\
s^{(j)}_1 & (j = 4, 5, 6, 7), \\
\end{array}
\right. \ 
m^{(j)}_2 := \left\{ 
\begin{array}{cl}
u_1 & (j=1), \\
s^{(j)}_2 & (j = 2, 3, 4, 5), \\
\left( u_j s^{(j)}_1 \right)^{\frac{1}{2}} & (j = 6, 7). \\
\end{array}
\right.
\end{eqnarray*}
By simple calculations, we can see them specifically as follows: 
\begin{Prop} \label{varphi}
Let $\tilde{\varphi}_j \left( v^{(j)}_0, v^{(j)}_1, v^{(j)}_2 \right) 
= \left( C_{(j)}, L_{(j)} \right)$ ($1 \leq j \leq 7$). Then 
\begin{align*}
C_{(1)} &: \left( v^{(1)}_0 \right )^3 x_0^3 + x_1^3 + x_2^3 = 3 x_0 x_1 x_2, \ \ 
L_{(1)} : x_0 + v^{(1)}_1 x_1 + v^{(1)}_2 x_2 = 0 ,
\\
% \end{align}
% \begin{align}
%
C_{(2)} &: x_1 (x_1^2 - 3 x_0 x_2) 
= - \left( v^{(2)}_1 \right)^{3/2} \left( \left( v^{(2)}_0 \right) ^3 x_0^3 + x_2^3 \right), \\
L_{(2)} &: x_0 + \left( v^{(2)}_1 \right) ^{1/2} v^{(2)}_2 x_1 + x_2 = 0 ,
\\
% \end{align*}
% \begin{align*}
%
C_{(3)} &: x_1 (x_1^2 - 3 x_0 x_2) = - \left( v^{(3)}_1 \right) ^3 \left( v^{(3)}_2 \right)^{3/2} 
\left( \left( v^{(3)}_0 \right) ^3 x_0^3 + x_2^3 \right), \\
L_{(3)} &: x_0 + \left( v^{(3)}_2 \right)^{1/2} x_1 + x_2 = 0 ,
%\\
 \end{align*}
 \begin{align*}
C_{(4)} &: 3 x_0 x_1 x_2 
= v^{(4)}_1 \left( x_1^3 + \left( v^{(4)}_2 \right) ^3 
\left( x_2^3 + \left( v^{(4)}_0 \right) ^3 x_0^3 \right) \right), \\
L_{(4)} &: x_0 + x_1 + x_2 = 0 
\end{align*}
and $\left( C_{(5)}, L_{(5)} \right)$ (resp. $\left( C_{(6)}, L_{(6)} \right)$ or 
$\left( C_{(7)}, L_{(7)} \right)$) is given by replacing $x_1$ with $x_2$, and 
$v^{(9-j)}_k$ with $v^{(j)}_k$ in $\left( C_{(4)}, L_{(4)} \right)$ 
(resp. $\left( C_{(3)}, L_{(3)} \right)$ or $\left( C_{(2)}, L_{(2)} \right)$). 
\end{Prop}

Note that $\left( C_{(2)}, L_{(2)} \right)$ 
(resp. $\left( C_{(3)}, L_{(3)} \right)$) is independent of a choice of a square root of $v_1^{(2)}$ (resp. $v_2^{(3)}$). 
In fact, we obtain same pairs in ${BP}_{1, 3}$ under the transformation $x_1 \mapsto - x_1$. 

We can show that the pairs in Proposition \ref{varphi} are contained in ${BP}_{1,3}$. 
Here we only consider in the case that $j=4$. 
The other cases can be checked by same arguments. 
If $v_1^{(4)} = 0$, then $C_{(4)}$ is a $3$-gon $V(x_0 x_1 x_2)$ and its singularities 
$$
p_0 := (1 : 0 : 0), \ p_1 := (0 : 1 : 0), \ p_2 := (0 : 0 : 1)
$$
are not contained in $L_{(4)} = V (x_0 + x_1 + x_2)$, and hence 
$\left( C_{(4)}, L_{(4)} \right) \in {BP}_{1,3}$. 
If $v_1^{(4)} \not = 0$ and $v_2^{(4)} = 0$, then we have 
$C_{(4)} = L' + Q$, $L' = V(x_1)$ and $Q = V\left( v^{(4)}_1 x_1^2 - 3 x_0 x_2 \right)$. 
Thus $\left( C_{(4)}, L_{(4)} \right) \in {BP}_{1,3}$, 
since $L_{(4)}$ does not pass through singular points $p_0$ and $p_2$ of $C_{(4)}$. 
If $v_1^{(4)} v_2^{(4)} \not = 0$ and $v_0^{(4)} = 0$, then 
$C_{(4)}$ is a nodal cubic 
$
V \left( v^{(4)}_1 \left( x_1^3 + \left( v^{(4)}_2 \right) ^3 x_2^3 \right) - 3 x_0 x_1 x_2
\right)
$ 
with the nodal point $p_0$ and $L_{(4)}$ does not pass through $p_0$. Hence 
$\left( C_{(4)}, L_{(4)} \right) \in {BP}_{1,3}$. 
If $v_0^{(4)} v_1^{(4)} v_2^{(4)} \not = 0$, then by the transformation 
$x_0 = y_0$, $x_1 = v^{(4)}_0 v^{(4)}_2 y_1$, $x_2 = v^{(4)}_0 y_2$, we obtain 
$
C_{(4)} : v^{(4)}_0 v^{(4)}_1 \left( v^{(4)}_2 \right)^2 (y_0^3 + y_1^3 + y_2^3 ) = 3 y_0 y_1 y_2
$ 
and 
$
L_{(4)} : y_0 + v^{(4)}_2 v^{(4)}_0 y_1 + v^{(4)}_0 y_2 = 0$. 
Then we have $\left( v^{(4)}_0 v^{(4)}_1 \left( v^{(4)}_2 \right)^2 \right)^3 \not = 1$ 
by the definition of $U_4$ and (\ref{u^(a,i)_j}). 
Thus $C_{(4)}$ is a nonsingular cubic curve, 
and hence $\left( C_{(4)}, L_{(4)} \right) \in {BP}_{1,3}$. 

\begin{Prop} \label{phi on U}
For any $j$, $j' \in \{ 1, \dots, 7\}$, we have that 
$\tilde{\varphi}_j = \tilde{\varphi}_{j'}$ on $U_j \cap U_{j'}$, and 
$\tilde{\varphi}_j = \varphi \circ \pi$ on $U_j \setminus \tilde{E}$, 
where $\tilde{E}$ is the exceptional set of $\pi$. 
In particular, there exists a morphism $\tilde{\varphi} : \tilde{U} \to BP_{1,3}$ such that 
$\varphi \circ \pi = \tilde{\varphi}$ on $\tilde{U} \setminus \tilde{E}$. 
\end{Prop}

\begin{proof}
By definition, 
$\tilde{\varphi}_j \left( v^{(j)}_0, v^{(j)}_1, v^{(j)}_2 \right) 
= M_j \cdot \left( V \left(F_{(j)} (x) \right), V \left( S_{(j)} (x) \right) \right)$ $(1 \leq j \leq 7)$. 
For each $j$, we can check easily that $M_j$ has nonzero determinant on 
$U_j \setminus \tilde{E}$. 
Moreover we have 
$\pi \left( v^{(j)}_0, v^{(j)}_1, v^{(j)}_2 \right) = \left( u_j, s^{(j)}_1, s^{(j)}_2 \right)$. 
Hence on $U_j \setminus \tilde{E}$, we have that 
\begin{align*}
BP_{1,3} \ni \tilde{\varphi}_j \left( v^{(j)}_0, v^{(j)}_1, v^{(j)}_2 \right) 
&= M_j^{-1} \cdot \tilde{\varphi}_j \left( v^{(j)}_0, v^{(j)}_1, v^{(j)}_2 \right) \\
&= \left( V \left(F_{(j)} (x) \right), V \left( S_{(j)} (x) \right) \right) \\
&= \varphi \left( (u_j : 1) \times \left( 1 : s^{(j)}_1 : s^{(j)}_2 \right) \right) \\
&= \varphi \circ \pi \left( v^{(j)}_0, v^{(j)}_1, v^{(j)}_2 \right). 
\end{align*}
Therefore we obtain 
$\tilde{\varphi}_j = \varphi \circ \pi$ on $U_j \setminus \tilde{E}$ ($1 \leq j \leq 7$). 
In particular, we obtain 
$\tilde{\varphi}_j = \tilde{\varphi}_{j'}$ on $\left( U_j \cap U_{j'} \right) \setminus \tilde{E}$ 
for any $j$ and $j'$. 

On the other hand, for any $j$ and $j'$, we can check easily that 
$\tilde{\varphi}_j = \tilde{\varphi}_{j'}$ on $U_j \cap U_{j'} \cap \tilde{E}$. 
For example, when $j = 1$ and $j' = 2$, the intersection $U_1 \cap U_2 \cap \tilde{E}$ 
is identified with 
\begin{align*}
\left. 
 U_1 \setminus V \left( v^{(1)}_2 \right) \middle| _{\left( v^{(1)}_0 = 0 \right)} 
\right. 
&\overset{\sim}{\longrightarrow}
\left. 
 U_2 \setminus V \left( v^{(2)}_2 \right) \middle| _{\left( v^{(2)}_0 = 0 \right)} 
\right. , \\ 
\left( 0, v^{(1)}_1, v^{(1)}_2 \right) 
&\longmapsto 
\left( 0, \frac{1}{v^{(1)}_2}, v^{(1)}_1 \right). 
\end{align*} 
Then we have 
\begin{align*}
\tilde{\varphi}_2 \left( 0, \frac{1}{v^{(1)}_2}, v^{(1)}_1 \right) 
&= \left\{
\begin{array}{l} 
C_{(2)} : x_1^3 + \left( v^{(1)}_2 \right)^{- 3/2} x_2^3 = 3 x_0 x_1 x_2, \\
L_{(2)} : x_0 + \frac{v^{(1)}_1}{\left( v^{(1)}_2 \right) ^{1/2}} x_1 + x_2 = 0 ,
\end{array} \right. \\
&= {\rm Diag} \left( 1, \left( v^{(1)}_2 \right)^{-1/2}, \left( v^{(1)}_2 \right)^{-1} \right) 
\cdot \tilde{\varphi}_1 \left( 0, v^{(1)}_1, v^{(1)}_2 \right). 
\end{align*}
Hence we have 
$\tilde{\varphi}_2 \left( 0, 1/v^{(1)}_2, v^{(1)}_1 \right) = \tilde{\varphi}_1 \left( 0, v^{(1)}_1, v^{(1)}_2 \right)$ in ${BP}_{1, 3}$, and hence we obtain 
$\tilde{\varphi}_1 = \tilde{\varphi}_2$ on $U_1 \cap U_2 \cap \tilde{E}$. 
Similarly, for the other pairs $(j, j')$, we can prove that $\tilde{\varphi}_j = \tilde{\varphi}_{j'}$ on $U_j \cap U_{j'} \cap \tilde{E}$. 
Hence we obtain 
$\tilde{\varphi}_j = \tilde{\varphi}_{j'}$ on $U_j \cap U_{j'}$. 
\end{proof}

{\bf Step 4.} 
We construct morphisms 
$\pi : \tilde{X} \to X$ and 
$\tilde{\varphi} : \tilde{X} \to BP_{1,3}$ such that 
$\pi$ is a composition of blowing-ups and 
$\tilde{\varphi}$ is an extension of $\varphi$. 
Similarly to Steps 2-3, for each affine open subset $U^{(a, i)}$ of $X$ 
($a = 0$ or $a^3 = 1$, and $0 \leq i \leq 2$), 
we can construct explicitly a morphism
$\pi^{(a, i)} : \tilde{U}^{(a, i)} \rightarrow U^{(a, i)}$ which is the composition 
of blowing-ups $\pi^{(a, i)}_1$, $\pi^{(a, i)}_2$ and $\pi^{(a, i)}_3$, and a morphism 
$\tilde{\varphi}^{(a, i)} : \tilde{U}^{(a, i)} \rightarrow {BP}_{1, 3}$ such that 
$\varphi \circ \pi^{(a, i)} = \tilde{\varphi}^{(a, i)}$ on $\tilde{U}^{(a, i)} \setminus \tilde{E}^{(a,i)}$, 
where $\tilde{E}^{(a, i)}$ is the exceptional set of $\pi^{(a, i)}$. 
We denote by $\pi : \tilde{X} \to X$ the scheme obtained by gluing the schemes $\pi^{(a, i)} : \tilde{U}^{(a, i)} \to U^{(a, i)}$ and $U^{\infty} = \bigcup U^{(\infty, i)}$, that is, 
$\pi : \tilde{X} \to X$ is the composition of blowing-ups 
\begin{align*} \label{blowing-up1}
X \overset{\pi_1}{\longleftarrow } 
B_{Z_1} (X) \overset{\pi_2}{\longleftarrow } 
B_{Z_2} B_{Z_1} (X) \overset{\pi_3}{\longleftarrow } 
B_{Z_3} B_{Z_2} B_{Z_1} (X) = \tilde{X}, 
\end{align*}
where 
$Z_1$ consists of $12$ origins $O^{(a, i)}$ of $U^{(a, i)}$ 
($a = 0$ or $a^3 = 1$, and $0 \leq i \leq 2$), 
$Z_2$ is the disjoint union of $12$ proper transforms of the base locus of $\varphi$ under $\pi_1$, and 
$Z_3$ is the disjoint union of $12$ lines obtained by gluing centers of blowing-ups 
$\pi^{(a, i)}_3$. 
By simple calculations, we can check that 
$\tilde{\varphi}^{(a, i)} = \tilde{\varphi}^{(a', i')}$ on $\tilde{U}^{(a, i)} \cap \tilde{U}^{(a', i')}$ 
for any $(a, i)$ and $(a', i')$. 
Moreover $\varphi$ is generically finite of degree $216$ by Lemma \ref{natural maps}, 
and $\pi$ is of degree $1$. 
Therefore we obtain 

\begin{Th} \label{main_1}
There exists 
a morphism $\tilde{\varphi} : \tilde{X} \to {BP}_{1,3}$ such that 
$\varphi \circ \pi = \tilde{\varphi}$ on $\tilde{X} \setminus \tilde{E}$, 
where $\tilde{E}$ is the exceptional set of $\pi : \tilde{X} \to X$. 
In particular, $\tilde{\varphi}$ is generically finite of degree $216$. 
\end{Th}

\begin{Rem} \label{structure of BP}
By the morphism $\tilde{\varphi} : \tilde{X} \rightarrow {BP}_{1, 3}$, 
we see a rough structure of ${BP}_{1, 3}$ in the following sense: 
By construction, each $\tilde{U}^{(a, i)}$ has an affine open covering 
$\tilde{U}^{(a, i)} = \bigcup_{j = 1}^{7} U^{(a, i)}_j$. 
Recall that $\pi^{(a, i)} : \tilde{U}^{(a, i)} \to U^{(a, i)}$ is a composition of blowing-ups 
$\pi^{(a, i)}_1$, $\pi^{(a, i)}_2$ and $\pi^{(a, i)}_3$. 
Let 
$E^{(a, i)}_\ell$ be the exceptional set of $\pi^{(a,i)} _\ell$ and 
$\tilde{E}^{(a, i)}_\ell$ (resp. $\tilde{E}^{(a, i)}_0$) be the proper transform of 
$E^{(a, i)}_\ell$ (resp. $V \left( u^{(a)} \right)$) under $\pi^{(a, i)}$. 
For example, when $(a, i) = (0, 0)$, 
\begin{gather*} 
\tilde{E}^{(0, 0)}_0 = V\left( v^{(4)}_1 \right) \cup V\left( v^{(5)}_1 \right), \ \ 
\tilde{E}^{(0, 0)}_1 = \bigcup_{j = 1}^{7} V\left( v^{(j)}_0 \right), 
\\
%\end{gather*}
%\begin{gather*} 
\tilde{E}^{(0, 0)}_2 = V\left( v^{(2)}_1 \right) \cup V\left( v^{(3)}_2 \right) \cup V\left( v^{(6)}_2 \right) \cup V\left( v^{(7)}_1 \right) , 
\\
%\end{gather*}
%\begin{gather*} 
\tilde{E}^{(0, 0)}_3 = V \left( v^{(3)}_1 \right) \cup V\left( v^{(4)}_2 \right) \cup V\left( v^{(5)}_2 \right) \cup V\left( v^{(6)}_1 \right), 
\end{gather*}
where $\left\{ v^{(j)}_0, v^{(j)}_1, v^{(j)}_2 \right\}$ is the coordinate of $U_j$ defined in Step 2. 
Let $\tilde{E}_{\ell} = \bigcup \tilde{E}^{(a, i)}_\ell $ ($0 \leq \ell \leq 3$). 
For any $v \in \tilde{X}$, we put $\tilde{\varphi} (v) = (C, L) \in {BP}_{1, 3}$. 
Then by simple calculations, 
\begin{align*} 
C \ {\rm is} \ \left\{ 
\begin{array}{ll} 
{\rm a \ } 3 {\rm \mathchar`- gon} & {\rm if} \ v \in \tilde{E}_0, \\
{\rm an \ irreducible \ conic \ plus \ a \ line} & {\rm if} \ v \in \left( \tilde{E}_2 \cup \tilde{E}_3 \right) \setminus \tilde{E}_0, \\
{\rm a \ nodal \ cubic} & {\rm if} \ v \in \tilde{E}_1 \setminus \left( \tilde{E}_0 \cup \tilde{E}_2 \cup \tilde{E}_3 \right) , \\
{\rm a \ smooth \ cubic} & {\rm if} \ v \in \tilde{X} \setminus \left( \tilde{E}_0 \cup \tilde{E}_1 \cup \tilde{E}_2 \cup \tilde{E}_3 \right) . 
\end{array} \right.
\end{align*}
In fact, this is clear on $\tilde{U} = \tilde{U}^{(0,0)}$ by Proposition \ref{varphi}. 
\end{Rem}

\subsection{Constructions of morphisms $p$ and $\hat{\psi}$} 

Recall that 
$\pi_1 : B_{Z_1} (X) \rightarrow X$ is the blowing-up at $Z_1$ consisting of $12$ origins $O^{(a, i)}$ of $U^{(a, i)}$ ($a = 0$ or $a^3=1$, and $0 \leq i \leq 2$). 
Thus by Lemma \ref{natural maps}, (iii), 
the proper transforms $\tilde{A}^{(j)}_i$ of $A^{(j)}_i$ under $\pi_1$ are disjoint each other and do not intersect with the center of the blowing-up 
$\pi_2 : B_{Z_2} B_{Z_1} (X) \rightarrow B_{Z_1} (X)$. 
Hence the proper transform of  each $A^{(j)}_i$ under $\pi : \tilde{X} \to X$ 
is isomorphic to $\tilde{A}^{(j)}_i$. 
By the proof of Lemma \ref{natural maps}, (ii), we have $\varphi^{-1} (W_{\rm T}) = \bigcup A_i^{(j)}$, where $W_{\rm T}$ is the subset of pairs $(C, L) \in BP_{1,3}$ consisting of a cubic curve $C$ and a line $L$ which is $3$-tangent to $C$. 
Hence we obtain  
\begin{align} \label{inverse image of WT}
\tilde{X} \supset \tilde{\varphi}^{-1} \left( W_{\rm T} \right) = \tilde{A} := 
\coprod_{0 \leq i, j \leq 2} \tilde{A}^{(j)}_i. 
\end{align}
On the other hand, by Lemma \ref{natural maps}, (ii), the base locus of $\psi$ consists of the center of $\pi$ and nine lines $A^{(j)}_i$, hence 
similarly to the definition of $\tilde{\varphi}$, we can define a rational map 
$\tilde{\psi} : \tilde{X} \rightarrow \overline{P}_{1, 3}$ with base locus $\tilde{A}$. 
Note that 
$\tilde{\varphi} = f \circ \tilde{\psi}$ on $\tilde{X} \setminus \tilde{A}$, 
where $f$ is the birational in Theorem \ref{main1}. 
In particular, 
$\psi \circ \pi = \tilde{\psi}$ on $\tilde{X} \setminus \left( \tilde{E} \cup \tilde{A} \right)$, 
and hence $\tilde{\psi}$ is generically finite of degree $216$. 

In the following Steps 1-4, we construct morphisms 
$p : \hat{X} \to \tilde{X}$ and $\hat{\psi} : \hat{X} \to \overline{P}_{1,3}$ 
such that $p$ is a composition of blowing-ups with the center $\tilde{A}$ and 
$\hat{\psi}$ is an extension of $\tilde{\psi}$. \\

{\bf Step 1.} 
We retake an affine open covering of $\tilde{X}$ such that 
the base locus of $\tilde{\psi}$ on each subset is empty or one axis. 
Recall that $\tilde{X}$ has the affine open covering 
$$
\tilde{X} = \bigcup _{0 \leq i \leq 2} U^{(\infty, i)} \cup \bigcup_{\substack{a=0 \ {\rm or} \ a^3=1, \\ 0 \leq i \leq 2, \ 1 \leq \ell \leq 7}} U^{(a, i)}_\ell .
$$
For $a = 0$ or $a^3 = 1$, and $i$, $j \in \{ 0, 1, 2\}$, let 
\begin{gather*}
\alpha ^{(a, i, j )}_1 = \frac{s^{(a)}_{[i+1], i}}{u^{(a)}} - \zeta_3^{2 j} \frac{s^{(a)}_{[i+2], i}}{u^{(a)}}, \ \ 
\alpha ^{(a, i, j)}_2 = \frac{s^{(a)}_{[i+2], i}}{u^{(a)}} - \zeta_3^{2 j} , \\ 
\alpha ^{(\infty, i, j)}_1 = s^{(\infty)}_{[i+1], i} - \zeta_3^j , \ \ \alpha ^{(\infty, i, j)}_2 = s^{(\infty)}_{[i+2], i} - \zeta_3^{2 j} u^{(\infty)} , 
\end{gather*}
where $[\ell] = \ell \mod 3 \in \{ 0, 1, 2 \}$. 
Then $U^{(a, i)}_1$ (resp. $U^{(\infty, i)}$) has a local coordinate 
$\left\{ u^{(a)}, \alpha ^{(a, i, j)}_1, \alpha ^{(a, i, j)}_2\right\}$
(resp. $\left\{ u^{(\infty)}, \alpha ^{(\infty, i, j)}_1, \alpha ^{(\infty, i, j)}_2 \right\}$). 
We define open subsets 
$V^{(a, i, j)}$ ($a=0$, $\infty$ or $a^3=1$, and $0 \leq i, j \leq 2$) and 
$V^{(a, i)}_\ell$ ($a=0$ or $a^3=1$, $0 \leq i \leq 2$, and $2 \leq \ell \leq 7$) of $\tilde{X}$ as follows: 
\begin{align*}
V^{(a, i, j)} &= {\rm Spec} \left( k \left[ u^{(a)}, \alpha ^{(a, i, j)}_1, \alpha ^{(a, i, j)}_2, \frac{1}{\left( u^{(a)} \right) ^3 - 1} \right] \right) 
\setminus {\mathcal C}^{(a, i, j)} , 
\\
%\end{align*} 
%\begin{align*}
V^{(a, i)}_\ell &= U^{(a, i)}_\ell \setminus \bigcup_{0 \leq s, t \leq 2} \tilde{A}^{(t)}_s ,
\end{align*}
where ${\mathcal C}^{(a, i, j)} := \bigcup_{(s, t) \not = (i_a , j_a)} \tilde{A}_s^{(t)} \subset \tilde{X}$ and $(i_a, j_a)$ is defined by 
\begin{eqnarray*} 
\left( i_a , j_a \right)  = \left\{
\begin{array}{l l} 
\left( i, j  \right) & (a = 0), \\
\left( j, [2 k (j + 1) + 2 i] \right) & (a = \zeta_3^k , \ 0 \leq k \leq 2), \\
\left( [i+2], j \right) & (a = \infty) . \\
\end{array} \right.
\end{eqnarray*}
Then $\tilde{X}$ has an affine open covering 
$$
\tilde{X} = 
\bigcup_{\substack{a \in \{ 0, 1, \zeta _3, \zeta_3^2, \infty \}, \\ 0 \leq i, j \leq 2}} V^{(a, i, j)}  
\cup \bigcup_{\substack{a=0 \ {\rm or} \ a^3=1, \\ 0 \leq i \leq 2, \ 2 \leq \ell \leq 7}} V^{(a, i)}_\ell ,
$$
and $\tilde{\psi}$ is well-defined on $\bigcup V^{(a, i)}_\ell$ and the base locus of 
$\tilde{\psi}$ on each $V^{(a, i, j)}$ is the $u^{(a)}$-axis, that is,  
$$
V^{(a, i, j)} \cap \tilde{A}^{(j_a)}_{i_a} = V \left( \alpha ^{(a, i, j)}_1, \alpha ^{(a, i, j)}_2 \right) \subset V^{(a, i, j)} .
$$ 
Note that these base loci are glued as follows:
\begin{eqnarray*} 
\tilde{A}^{(j)}_i 
=
V \left( \alpha ^{(0, i, j)}_1, \alpha ^{(0, i, j)}_2 \right) 
&\cup & 
\bigcup_{0 \leq k \leq 2} 
V \left( \alpha ^{(\zeta_3^k, [2 j + (2 - i) k], i)}_1, \alpha ^{(\zeta_3^k, [2 j + (2 - i) k], i)}_2 \right) \nonumber \\
&\cup & V \left( \alpha ^{(\infty, [i+1], j)}_1, \alpha ^{(\infty, [i+1], j)}_2 \right) .
\end{eqnarray*}

{\bf Step 2.} 
We blow up each $V^{(a, i, j)}$ along the base locus of $\tilde{\psi}$. 
We give a construction in detail in the case of $V^{(0,0,0)}$, since the other cases are similar. 
For simplicity, we put $(u, \alpha_1, \alpha_2) = 
(u^{(0)}, \alpha^{(0,0,0)} _1, \alpha^{(0,0,0)} _2)$ 
and $V = V^{(0,0,0)}$. 
Note that by using notations in Subsection \ref{to BP}, we have 
$V \subset U_1$ and $\alpha _1 = \frac{s_1}{u} - \frac{s_2}{u}$, 
$\alpha_2 = \frac{s_2}{u} - 1$. 
In particular, by Proposition \ref{varphi}, $\tilde{\psi}$ is given by 
\begin{gather*}
\tilde{\psi} (u, \alpha_1, \alpha_2) = 
\left( V \left( \tilde{F} (x) \right), V \left( \tilde{S} (x) \right) \right), \ 
\tilde{F} (x) := u^3 x_0^3 + x_1^3 + x_2^3 - 3 x_0 x_1 x_2, \label{general F} \\
\tilde{S} (x) := x_0 + (\alpha_1 + \alpha_2 + 1) x_1 + (\alpha_2 + 1) x_2 \label{general S} 
\end{gather*}
on $V$, and its base locus is $V (\alpha_1, \alpha_2)$. 
We blow up $V$ as follows: 
\begin{equation*} 
V \overset{p _1}{\longleftarrow } 
B_{W_1} (V)  \overset{p _2}{\longleftarrow } 
B_{W_2} B_{W_1} (V) \overset{p _3}{\longleftarrow } 
\hat{V} := B_{W_3} B_{W_2} B_{W_1} (V).
\end{equation*}
We denote by $p$ the composition of three blowing-ups $p_1$, $p_2$ and $p_3$. 
The centers $W_1$, $W_2$ and $W_3$ are defined as follows: 
\begin{itemize}
\item $W_1$ is the base locus of $\tilde{\psi}$ on $V$, that is, the line $V (\alpha_1, \alpha_2)$. 
\item $W_2$ is a line $L'$ in the exceptional set of $p_1$. 
Here $L'$ is defined as follows: 
$B_{W_1} (V)$ has an affine open covering consisting of $2$ affine open subsets defined by 
\begin{align*}
{\rm Spec} 
\left( k \left[ 
u, \alpha _1, \frac{\alpha _2}{\alpha _1}, \frac{1}{u^3 - 1} 
\right] \right) \setminus {\mathcal C}', \ 
{\rm Spec} \left( k \left[ 
u, \frac{\alpha_1}{\alpha_2}, \alpha_2, \frac{1}{u^3 - 1} 
\right] \right)  \setminus {\mathcal C}', 
\end{align*}
where ${\mathcal C}'$ is the proper transform of ${\mathcal C}^{(0,0,0)}$ 
under $p_1$. 
Then we define $L'$ as the $u$-axis in the second affine subset, that is, 
$L' = V \left( \frac{\alpha_1}{\alpha_2}, \alpha_2 \right)$. 
\item $W_3$ is a line $L''$ in the exceptional set of $p_2$. 
Here $L''$ is defined as follows: 
$B_{W_2} B_{W_1} (V)$ has an affine open covering consisting of $3$ affine open subsets defined by 
\begin{gather*}
{\rm Spec} \left( k \left[ 
u, \alpha_1, \frac{\alpha_2}{\alpha_1}, \frac{1}{u^3 - 1} 
\right] \right) \setminus {\mathcal C}'', \ 
{\rm Spec} \left( k \left[ 
u, \frac{\alpha_1}{\alpha_2}, \frac{\alpha_2^2}{\alpha_1}, \frac{1}{u^3 - 1} 
\right] \right) \setminus {\mathcal C}'', \\
{\rm Spec} \left( k \left[ 
u, \frac{\alpha_1}{\alpha_2^2}, \alpha_2, \frac{1}{u^3 - 1} 
\right] \right)  \setminus {\mathcal C}'', 
\end{gather*}
where ${\mathcal C}''$ is the proper transform of 
${\mathcal C}'$ under $p_2$. 
Then we define $L''$ as the $u$-axis in the second affine subset, that is, 
$L'' = V \left( \frac{\alpha_1}{\alpha_2} , \frac{\alpha_2^2}{\alpha_1} \right)$. 
\end{itemize}
Let $\hat{{\mathcal C}}$ be the proper transform of 
${\mathcal C}''$ under $p_3$. 
Then $\hat{V} = B_{W_3} B_{W_2} B_{W_1} (V)$ has an affine open covering consisting of $4$ 
affine open subsets defined by 
\begin{align*}
V_1 
&=
{\rm Spec} \left( k \left[ u, \alpha_1, \frac{\alpha_2}{\alpha_1}, \frac{1}{u^3 - 1} \right] \right) \setminus \hat{{\mathcal C}}, 
V_2 
=
{\rm Spec} \left( k \left[ 
u, \frac{\alpha_1}{\alpha_2}, \frac{\alpha_2^3}{\alpha_1^2}, \frac{1}{u^3 - 1} \right] \right) 
\setminus \hat{{\mathcal C}}, 
\\
%\end{align*}
%\begin{align*}
V_3 
&=
{\rm Spec} \left( k \left[ 
u, \frac{\alpha_1^2}{\alpha_2^3}, \frac{\alpha_2^2}{\alpha_1}, \frac{1}{u^3 - 1} \right] \right)  \setminus \hat{{\mathcal C}},  
V_4 
=
 {\rm Spec} \left( k \left[ 
u, \frac{\alpha_1}{\alpha_2^2}, \alpha_2, \frac{1}{u^3 - 1} \right] \right)  
\setminus \hat{{\mathcal C}} . 
\end{align*}
In what follows, we denote by $w^{(r)}_0$, $w^{(r)}_1$, $w^{(r)}_2$ the above coordinates of $V_r$. 

{\bf Step 3.} 
We define explicitly a morphism 
$\hat{\psi}$ from $\hat{V}$ to $\overline{P}_{1, 3}$ such that $\hat{\psi}$ is an extension of $\tilde{\psi}$. 
For it, we define morphisms $\hat{\psi}_r$ from $V_r$ to $\overline{P}_{1,3}$, respectively, and we glue these morphisms. 
For each $r \in \{ 1, \dots, 4\}$, let 
\begin{gather*}
F^{(r)} (x) = \left( w^{(r)}_0 \right)^3 x_0^3 + x_1^3 + x_2^3  - 3 x_0 x_1 x_2, \\
S^{(r)} (x) = x_0 + \left( \alpha^{(r)}_1 + \alpha^{(r)}_2 + 1 \right) x_1 
+ \left( \alpha^{(r)}_2 + 1 \right) x_2, \ \ \mbox{where}, \\
%\end{gather*}
%\begin{align*}
\alpha^{(r)}_1 = \left\{ 
\begin{array}{cc}
w^{(1)}_1 & (r = 1), \\
\left( w^{(2)}_1 \right)^{3} w_2^{(2)} & (r=2),\\
\left( w^{(3)}_1 \right)^{2} \left( w^{(3)}_2 \right)^{3} & (r=3), \\
w^{(4)}_1 \left( w^{(2)}_2 \right)^{2} & (r = 4), 
\end{array}
\right. 
\alpha^{(r)}_2 = \left\{ 
\begin{array}{cc}
w^{(1)}_1 w^{(1)}_2 & (r = 1), \\
\left( w^{(2)}_1 \right)^{2} w_2^{(2)} & (r=2),\\
w^{(3)}_1 \left( w^{(3)}_2 \right)^{2} & (r=3), \\
w^{(4)}_2 & (r = 4). 
\end{array}
\right. 
\end{gather*}
Then we define $\hat{\psi}_r$ formally as follows: 
\begin{align*}
\hat{\psi}_r \left( w^{(r)}_0, w^{(r)}_1, w^{(r)}_2 \right) 
&:= N_r \cdot \left( V \left( F^{(r)} (x) \right), V \left( S^{(r)} (x) \right) \right) \\
&= \left( V \left( F^{(r)} (N_r ^{-1} \cdot x) \right), V \left( S^{(r)} (N_r ^{-1} \cdot x) \right) \right), 
\end{align*}
where each $N _{r}$ is defined by 
\begin{align*}
N _{r} 
= 
\left\{
\begin{array}{l l} 
\begin{bmatrix}
1 & 1 & 1 \\
0 & \alpha^{(r)}_1 & 0 \\
0 & \left( \alpha^{(r)}_1 \right)^{\frac{2}{3}} & \left( \alpha^{(r)}_1 \right)^{\frac{2}{3}} \\
\end{bmatrix}  \ (r = 1, 2), \\
 & \\
\begin{bmatrix}
1 & 1 & 1 \\
0 & \left( \alpha^{(r)}_2 \right)^{\frac{3}{2}} & 0 \\
0 & \alpha^{(r)}_2 & \alpha^{(r)}_2 \\
\end{bmatrix}  \ (r = 3, 4). 
\end{array} \right. 
\end{align*}
By simple calculations, we can see them specifically as follows: 
\begin{Prop} \label{psi} 
Let $\hat{\psi}_r \left(w^{(r)}_0, w^{(r)}_1, w^{(r)}_2\right) 
= \left(C^{(r)}, L^{(r)} \right)$ ($1 \leq r \leq 4$). 
Then we have that 
\begin{align*}
C^{(1)} &: \left( \left( w_0^{(1)} \right) ^3 - 1 \right) x_{2}^3 = 3 x_0 x_{1}^2 - 3 \left( w^{(1)}_1 \right) ^{\frac{1}{3}} x_0 x_1 x_2  \\
&\hspace{35pt}
+ \left( w^{(1)}_0 \right) ^3 x_{0} \left( \left( w^{(1)}_1 \right) ^2 x_{0}^2 
- 3 \left( w^{(1)}_1 \right) ^{\frac{4}{3}} x_{0} x_{2} \right. 
\left. + 3 \left( w^{(1)}_1 \right) ^{\frac{2}{3}} x_{2}^2 \right) , \\
L^{(1)} &: x_0 + x_1 + \left( w^{(1)}_1 \right) ^{\frac{1}{3}} w^{(1)}_2 x_{2} = 0 , 
\\
%\end{align*}
%\begin{align*}
C^{(2)} &: \left( \left( w^{(2)}_0 \right) ^3 - 1 \right) x_{2}^3 = 3 x_0 x_{1}^2 - 3 w^{(2)}_1 \left( w^{(2)}_2 \right) ^{\frac{1}{3}} x_0 x_1 x_2 + \left( w^{(2)}_0 \right)^3 \left( w^{(2)}_1 \right)^2 x_{0} \\
& \hspace{35pt} \cdot 
\left( \left( w^{(2)}_1 \right)^4 \left( w^{(2)}_2 \right)^2 x_0^2 \right. 
\left. - 3 \left( w^{(2)}_1 \right)^2 \left( w^{(2)}_2 \right)^{\frac{4}{3}} x_{0} x_2 + 3 \left( w^{(2)}_2 \right)^{\frac{2}{3}} x_{2}^2 \right) , \\
L^{(2)} &: x_0 + x_{1} + \left( w^{(2)}_2 \right)^{\frac{1}{3}} x_{2} = 0 ,
%\\
\end{align*}
\begin{align*}
C^{(3)} &: \left( \left( w^{(3)}_0 \right)^3 - 1 \right) x_{2}^3 
= 3 x_0 x_{1}^2 
- 3 \left( w^{(3)}_1 \right)^{\frac{1}{2}} w^{(3)}_2 x_0 x_1 x_2 
+ \left( w^{(3)}_0 \right)^3 w^{(3)}_1  \\
&\hspace{40pt} 
\cdot \left( w^{(3)}_2 \right)^2 x_0 
\left( 
\left( w^{(3)}_1 \right) ^2 \left( w^{(3)}_2 \right) ^4 x_0^2 
- 3 w^{(3)}_1 \left( w^{(3)}_2 \right) ^{2} x_{0} x_2 + 3 x_{2}^2 
\right) , \\
L^{(3)} &: x_0 + \left( w^{(3)}_1 \right) ^{\frac{1}{2}} x_{1} + x_{2} = 0 ,
\\
% \end{align*}
% \begin{align*}
C^{(4)} &: \left( \left( w^{(4)}_0 \right)^3 - 1 \right) x_{2}^3 
= 3 x_0 x_{1}^2 
- 3 \left( w^{(4)}_2 \right)^{\frac{1}{2}} x_0 x_1 x_2 \\
&\hspace{40pt} 
+ \left( 
w^{(4)}_0 \right)^3 w^{(4)}_2 x_0 \left( \left( w^{(4)}_2 \right)^2 x_0^2 
- 3  w^{(4)}_2 x_{0} x_2 + 3 x_{2}^2 
\right) , 
\\
% \end{align*}
% \begin{align*}
L^{(4)} &: x_0 + w^{(4)}_1 \left( w^{(4)}_2 \right) ^{\frac{1}{2}} x_{1} + x_{2} = 0 .
\end{align*}
\end{Prop}

Note that $\left(C^{(1)}, L^{(1)}\right)$ and $\left(C^{(3)}, L^{(3)}\right)$ 
(resp. $\left(C^{(2)}, L^{(2)}\right)$ and $\left(C^{(4)}, L^{(4)}\right)$) 
are independent of choices of a square root and a cubic root of 
$w^{(r)}_1$ (resp. $w^{(r)}_2$). 
In fact, when $r=1$ or $2$ (resp. $r=3$ or $4$), 
we have same semistable pairs in $\overline{P}_{1, 3}$ under the transformation $x_2 \mapsto \beta x_2$, $\beta^3=1$ (resp. $x_1 \mapsto - x_1$). 

We can show that each pair $\left(C^{(r)}, L^{(r)}\right)$
is samistable. 
Here we only consider in the case that $r=3$.  
The other cases can be checked by same arguments. 
If $w^{(3)}_1 w^{(3)}_2 = 0$, then we have 
$$
C^{(3)} : \left( \left( w^{(3)}_0 \right)^3 - 1 \right) x_2^3 = 3 x_0 x_1^2 .
$$
By definition, we have $\left( w^{(3)}_0 \right)^3 \not = 1$, hence $C^{(3)}$ is a cuspidal curve with the cusp point $p_0 = (1 : 0 : 0)$. 
Then $L^{(3)}$ does not pass through $p_0$ and $L$ is not $3$-tangent to $C^{(3)}$. 
Hence $\left( C^{(3)}, L^{(3)} \right)$ is a semistable pair. 
Suppose $w^{(3)}_1 w^{(3)}_2 \not = 0$. Then by the transformation 
$$
x_0 = \left( y_0 + y_1 + y_2 \right) / \left( w^{(3)}_1 \left( w^{(3)}_2 \right)^2 \right), \ 
x_1 = \left( w^{(3)}_1 \right)^{1/2} w^{(3)}_2 y_1, \ 
x_2 = y_2 + y_1,
$$
we obtain $C^{(3)} :  \left( w^{(3)}_0 \right) ^3 y_0^3 + y_1^3 + y_2^3 = 3 y_0 y_1 y_2$ and 
\begin{equation*}
L^{(3)} : y_0 + \left( w^{(3)}_1 \left( w^{(3)}_2 \right)^2 \left( w^{(3)}_1 w^{(3)}_2 + 1 \right) + 1 \right) y_1 
+ \left( w^{(3)}_1 \left( w^{(3)}_2 \right)^2 + 1 \right) y_2 = 0 .
\end{equation*}
If $w_0^{(3)} = 0$, then $C^{(3)}$ is a nodal curve with the nodal point $p_0$, 
and $L^{(3)}$ does not pass through this point. 
Hence $\left( C^{(3)}, L^{(3)} \right)$ is a semistable pair. 
If $w_0^{(3)} \not = 0$, then $C^{(3)}$ is a smooth curve, 
since $\left( w^{(3)}_0 \right) ^3 \not = 1$. 
Recall that $L^{(3)}$ is a $3$-tangent to $C^{(3)}$ if and only if 
$\left(w^{(3)}_0, w^{(3)}_1, w^{(3)}_2\right) \in \hat{{\mathcal C}}$. 
By construction, we have $V_3 \cap \hat{{\mathcal C}} = \emptyset$. 
Hence $L^{(3)}$ is not a $3$-tangent to $C^{(3)}$. 
Thus $\left( C^{(3)}, L^{(3)} \right)$ is a semistable pair. 

By similar arguments in the proof of Proposition \ref{phi on U}, 
we can prove the following Proposition: 
\begin{Prop}
For any $r$, $r' \in \{ 1, \dots, 4\}$, we have that 
$\hat{\psi}_r = \hat{\psi}_{r'}$ on $V_r \cap V_{r'}$, and 
$\hat{\psi}_r = \tilde{\psi} \circ p$ on $V_r \setminus \hat{E}$, 
where $\hat{E}$ is the exceptional set of $p$. 
In particular, there exists a morphism $\hat{\psi} : \hat{V} \to \overline{P}_{1,3}$ such that 
$\tilde{\psi} \circ p = \hat{\psi}$ on $\hat{V} \setminus \hat{E}$. 
\end{Prop}

{\bf Step 4.} 
We construct morphisms 
$p : \hat{X} \to \tilde{X}$ and 
$\hat{\psi} : \hat{X} \to \overline{P}_{1,3}$ such that 
$p$ is a composition of blowing-ups and $\hat{\psi}$ is an extension of $\tilde{\psi}$. 
Similarly to Steps 2-3, for each affine open subset $V^{(a, i, j)}$ of $\tilde{X}$ 
($a=0$, $\infty$ or $a^3=1$, and $0 \leq i, j \leq 2$), 
we can construct explicitly a morphism 
$p^{(a, i, j)} : \hat{V}^{(a, i, j)} \to V^{(a, i, j)}$ which is the composition of blowing-ups 
$p^{(a, i, j)}_1$, $p^{(a, i, j)}_2$ and $p^{(a, i, j)}_3$, 
and a morphism $\hat{\psi}^{(a, i, j)} : \hat{V}^{(a, i, j)} \to \overline{P}_{1,3}$ such that 
$\tilde{\psi} \circ p^{(a, i, j)} = \hat{\psi}^{(a,i, j)}$ on 
$\hat{V}^{(a, i, j)} \setminus \hat{E}^{(a, i, j)}$, 
where $\hat{E}^{(a, i, j)}$ is the exceptional set of $p^{(a, i, j)}$. 
We denote by $p : \hat{X} \to \tilde{X}$ the scheme obtained by gluing the schemes 
$p^{(a, i, j)} : \hat{V}^{(a, i, j)} \to V^{(a, i, j)}$ and $\bigcup V^{(a, i)}_\ell$, that is, 
$p : \hat{X} \to \tilde{X}$ is the composition of blowing-ups 
\begin{align*} \label{blowing-up2}
\tilde{X} \overset{p_1}{\longleftarrow}
B_{W_1} \tilde{X} \overset{p_2}{\longleftarrow}
B_{W_2} B_{W_1} (\tilde{X}) \overset{p_3}{\longleftarrow}
B_{W_3} B_{W_2} B_{W_1} (\tilde{X}) = \hat{X}, 
\end{align*}
where for each $\ell$, 
$W_\ell$ is the disjoint union of $9$ lines obtained by gluing centers of blow-ups 
$p_\ell^{(a, i, j)}$ 
($a=0$, $\infty$ or $a^3=1$, and $0 \leq i, j \leq 2$). 
Note that $W_1$ is $\tilde{A}$ defined by (\ref{inverse image of WT}). 
By simple calculations, we can check that $\hat{\psi} ^{(a, i, j)} = \hat{\psi}^{(a', i', j')}$ 
on $\hat{V}^{(a, i, j)} \cap \hat{V}^{(a', i', j')}$ for any $(a, i, j)$ and $(a', i', j')$. 
Recall that $\tilde{\psi}$ is generically finite of degree $216$ and $p$ is of degree $1$. 
Therefore we obtain 

\begin{Th} \label{main_2}
There exists a morphism $\hat{\psi} : \hat{X} \rightarrow \overline{P}_{1,3}$ such that 
$\tilde{\psi} \circ p = \hat{\psi}$ on $\hat{X} \setminus \hat{E}$, where 
$\hat{E}$ is the the exceptional set of $p : \hat{X} \rightarrow \tilde{X}$. 
In particular, $\hat{\psi}$ is generically finite of degree $216$. 
\end{Th}

\begin{Rem}
By the morphism $\hat{\psi} : \hat{X} \rightarrow \overline{P}_{1, 3}$, we see a rough structure of $\overline{P}_{1, 3}$ in the following sense: 
By construction, each $\hat{V}^{(a, i, j)}$ has an affine open covering $\hat{V}^{(a, i, j)} 
= \bigcup_{r=1}^4 V^{(a, i, j)}_r$. 
Recall that $p^{(a, i, j)} : \hat{V}^{(a, i, j)} \to V^{(a, i, j)}$ is a composition of blowing-ups 
$p^{(a, i, j)}_1$, $p^{(a, i, j)}_2$ and $p^{(a, i, j)}_3$. 
Let $E^{(a, i, j)}_\ell$ be the exceptional set of $p^{(a, i, j)}_\ell$ and 
$\hat{E}^{(a, i, j)}_\ell$ be the proper transform of $E^{(a, i, j)}_\ell$ under $p^{(a, i, j)}$. 
For example, when $(a, i, j) = (0, 0, 0)$, 
\begin{gather*} 
\hat{E}^{(0,0,0)}_1 = V \left( w^{(1)}_1 \right) \cup V \left( w^{(2)}_2 \right), \ \ 
\hat{E}^{(0,0,0)}_2 = V \left( w^{(3)}_1 \right) \cup V \left( w^{(4)}_2 \right), \\
\hat{E}^{(0,0,0)}_3 = V \left( w^{(2)}_1 \right) \cup V \left( w^{(3)}_2 \right), 
\end{gather*}
where $\left\{ w^{(r)}_0, w^{(r)}_1, w^{(r)}_2 \right\}$ is the coordinate of $V_r$ defined in Step 2. 
Let $\hat{E}_\ell = \bigcup \tilde{E}^{(a, i, j)}_{\ell}$ ($1 \leq \ell \leq 3$). 
For any $w \in \hat{X}$, we put $\hat{\psi} (w) = (C, L) \in \overline{P}_{1, 3}$. 
Then by simple calculations, 
\begin{align} \label{inverse image of Cusps}
\mbox{
$C$ is a cuspidal cubic if and only if $w \in \hat{E} = \bigcup_{\ell=1}^3 \hat{E}_\ell$. 
}
\end{align}
In fact, this is clear on $\hat{V} = \hat{V}^{(0,0,0)}$ by Proposition \ref{psi}. 
On the other hand, 
when $w \in \hat{X} \setminus \hat{E}$, 
we have same results as Remark \ref{structure of BP}, since $\hat{X} \setminus \hat{E} \simeq \tilde{X} \setminus \tilde{A}$.   
\end{Rem}

\subsection{Proof of Proposition \ref{main2}}

The first assertion (i) is clear from Theorems \ref{main_1} and \ref{main_2}. 
The second assertion (ii) is clear from definitions (see Section \ref{comparison}). 
We prove the third assertion (iii). 
By (\ref{inverse image of Cusps}), we obtain that $\hat{\psi}^{-1} (W_{\rm Cusp}) = \hat{E}$, 
where $W_{\rm Cusp}$ is the subset of pairs $(C, L) \in \overline{P}_{1,3}$ consisting of $C$ is a cuspidal cubic.  
On the other hand, we have (\ref{inverse image of WT}), and hence 
we obtain (iii). 
\section*{Acknowledgements}

This paper is based on the author's doctoral thesis. 
The author would like to express his appreciation to Professor Iku Nakamura for suggesting this topic and for valuable advices and encouragement during the preparation of this paper. 
He has kindly shown his many manuscripts to the author. 
In particular, Section 3 is mainly due to his contributions. 
Of course, all of this paper is responsible to the author.

\smallskip
Masamichi Kuroda \\
Department of Mathematics \\
Hokkaido University \\
Sapporo 060-0810 \\
Japan 

\smallskip
\hspace{-0.5cm}m-kuroda@math.sci.hokudai.ac.jp

\end{document}